\documentclass[preprint,12pt]{elsarticle}
\usepackage{mathtext}
\usepackage{amsmath}
\usepackage{amsthm}

\usepackage{amssymb}

\newcommand{\R}{\mathbb{R}}
\newcommand{\N}{\mathbb{N}}

\newcommand{\CO}{C^{1,\alpha}(\partial\Omega)}
\newcommand{\COo}{C^{1,\alpha}(\partial\Omega^o)}
\newcommand{\COi}{C^{1,\alpha}(\partial\Omega^i)}

\newcommand{\COeiharm}{C^{1,\alpha}_{\mathrm{harm}}(\epsilon \overline{\Omega^i})}

\newtheorem{teo}{Theorem}[section]
\newtheorem{defin}[teo]{Definition}

\newtheorem{prop}[teo]{Proposition}
\newtheorem{lemma}[teo]{Lemma}
\newtheorem{cor}[teo]{Corollary}

\allowdisplaybreaks
\makeatletter
\def\ps@pprintTitle{%
	\let\@oddhead\@empty
	\let\@evenhead\@empty
	\let\@oddfoot\@empty
	\let\@evenfoot\@oddfoot
}
\makeatother

\begin{document}
\date{}
\begin{frontmatter}



\title{Local uniqueness of the solutions for a singularly perturbed nonlinear nonautonomous transmission problem}


\author[MDR]{Matteo Dalla Riva}
\author[RM]{Riccardo Molinarolo}
\author[PM]{Paolo Musolino}

\address[MDR]{Department of Mathematics, The University of Tulsa, 800 South Tucker Drive, Tulsa, Oklahoma 74104, USA. matteo-dallariva@utulsa.edu}
\address[RM]{Department of Mathematics, Aberystwyth University, Ceredigion SY23 3BZ, Wales, United Kingdom. rim22@aber.ac.uk}
\address[PM]{Dipartimento di Matematica `Tullio Levi-Civita', Universit\`a degli Studi di Padova, Via Trieste 63, Padova 35121, Italy. musolino@math.unipd.it}

\begin{abstract}
We consider the Laplace equation in a domain of $\mathbb{R}^n$, $n\ge 3$, with a small inclusion of size $\epsilon$. On the boundary of the inclusion we define a nonlinear nonautonomous transmission condition. For $\epsilon$ small enough one can prove that the problem has solutions.  In this paper, we study the local uniqueness of such solutions. 
\end{abstract}



\end{frontmatter}


\section{Introduction}

We study local uniqueness  properties of the solutions of a nonlinear nonautonomous transmission problem for the Laplace equation in the pair of sets  consisting of a perforated domain and a small inclusion.

We begin by presenting the geometric framework of our problem. We fix once for all a natural number 
\[
n\ge 3
\]
that will be  the dimension of the space $\mathbb{R}^n$ we are going to work in and a parameter
\[
\alpha \in ]0,1[
\]
which we use to define the regularity of our sets and functions. We remark that the case of dimension $n=2$ requires specific techniques and it is not treated in this paper (the analysis for $n=3$ and for $n\ge 3$ is instead very similar).

Then, we introduce two sets $\Omega^o$ and $\Omega^i$ that satisfy the following conditions:
\[
	\begin{split}
		&\mbox{$\Omega^o$, $\Omega^i$ are bounded open connected subsets of $\R^n$ of class $C^{1,\alpha}$,} 
		\\
		&\mbox{their exteriors  $\R^n\setminus \overline{\Omega^o}$ and $\R^n\setminus \overline{\Omega^i}$ are connected,}
		\\
		&\mbox{and the origin  $0$ of $\mathbb{R}^n$ belongs both to $\Omega^o$ and to $\Omega^i$.}
	\end{split}
\]
Here the superscript  ``$o$" stands for ``outer domain" whereas the superscript ``$i$"  stands for ``inner domain". We take
\[
\epsilon_0 \equiv \mbox{sup}\{\theta \in ]0, +\infty[: \epsilon \overline{\Omega^i} \subseteq \Omega^o, \ \forall \epsilon \in ]- \theta, \theta[  \}, 
\]
and we define the perforated domain $\Omega(\epsilon)$ by setting
\[
\Omega(\epsilon) \equiv \Omega^o \setminus \epsilon \overline{\Omega^i}
\]
for all $\epsilon\in]-\epsilon_0,\epsilon_0[$. Then we fix three functions
\begin{equation}\label{FGf^o}
\begin{split}
&F \colon ]-\epsilon_0,\epsilon_0[ \times \partial \Omega ^i \times \R \to \R\,,\\
&G \colon ]-\epsilon_0,\epsilon_0[ \times \partial \Omega ^i \times \R \to \R\,,\\
&f^o \in \COo
\end{split}
\end{equation}
and, for $\epsilon \in ]0,\epsilon_0[$, we consider the following nonlinear nonautonomous transmission problem in the perforated domain $\Omega(\epsilon)$ for a pair of functions $(u^o,u^i)\in C^{1,\alpha}(\overline{\Omega(\epsilon)})\times C^{1,\alpha}(\overline{\epsilon\Omega^i})$:
\begin{equation}\label{princeq}
	\begin{cases}
		\Delta u^o = 0 & \mbox{in } \Omega(\epsilon), \\
		\Delta u^i = 0 & \mbox{in } \epsilon \Omega^i, \\
		u^o(x)=f^o(x) & \forall x \in \partial \Omega^o, \\
		u^o(x) = F\left(\epsilon,\frac{x}{\epsilon},u^i(x)\right) & \forall x \in \epsilon \partial \Omega^i, \\
		\nu_{\epsilon \Omega^i} \cdot \nabla u^o (x) - \nu_{\epsilon \Omega^i} \cdot \nabla u^i (x) = G\left(\epsilon, \frac{x}{\epsilon},u^i(x)\right) & \forall x \in \epsilon \partial \Omega^i.
	\end{cases}
\end{equation}
Here $\nu_{\epsilon\Omega^i}$ denotes the outer exterior normal to $\epsilon\Omega^i$.

Boundary value problems like \eqref{princeq} arise in the mathematical models for the heat conduction in (nonlinear) composite materials, as, for example, in Mishuris, Miszuris and \"Ochsner \cite{MiMiOc07,MiMiOc08}. In such cases, the functions $u^o$ and $u^i$ represent the temperature distribution in $\Omega(\epsilon)$ and in the inclusion $\epsilon \Omega^i$, respectively. The third condition in \eqref{princeq} means that we are prescribing the temperature distribution on the exterior boundary $\partial \Omega^o$. The fourth condition says that on the interface $\epsilon \partial \Omega^i$ the temperature distribution $u^o$ depends nonlinearly on the size of the inclusion, on the position on the interface, and on the temperature distribution $u^i$. The fifth conditions, instead, says that the jump of the heat flux on the interface depends nonlinearly on the size of the inclusion, on the position on the interface, and on the temperature distribution $u^i$.

In literature, existence and uniqueness of solutions of nonlinear boundary value problems have been largely investigated by means of variational techniques (see, e.g., the monographs of Ne\v{c}as \cite{Ne83} and of Roub\'i\v{c}ek \cite{Ro13} and the references therein). On the other hand, boundary value problems in singularly perturbed domains are usually studied by expansion methods of asymptotic analysis, such as the methods of matching inner and outer expansions (cf., e.g., Il'in \cite{Il78, Il92}) and the multiscale expansion method (as in Maz'ya, Nazarov, and Plamenenvskii \cite{MaNaPl00}, see also Iguernane et al.~\cite{IgNaRoSoSz09} in connection with nonlinear problems). 

In this paper  we do not use variational techniques and  neither we resort to asymptotic expansion methods. Instead, we adopt the functional analytic approach proposed by Lanza de Cristoforis (cf., e.g., \cite{La02,La08}). 
The key strategy of such approach is the transformation of the perturbed boundary value problem into a functional equation that can be studied by means of certain theorems of functional analysis, for example by the implicit function theorem or by the Schauder fixed-point theorem. Typically, this transformation is achieved with an integral representation of the solution, a suitable rescaling of the corresponding integral equations,  and an analysis based on results of potential theory. This approach has proven to be effective when dealing with nonlinear conditions on the boundary of small holes or inclusions. For example, it has been used in the papers \cite{La07} and \cite{La10} of Lanza de Cristoforis  to study a nonlinear Robin and a nonlinear transmission problem, respectively,  in the paper \cite{DaLa10} with Lanza de Cristoforis to analyze a nonlinear traction problem for an elastic body, in \cite{DaMi15} with Mishuris to prove the existence of solution for \eqref{princeq} in the case of a ``big" inclusion (that is, for $\epsilon>0$ fixed), and in \cite{Mo18} to show the existence of solution of \eqref{princeq} in the case of  ``small" inclusion (that is, for $\epsilon>0$ that tends to $0$).

In particular, in  \cite{Mo18} we have proven that, under suitable assumptions on the functions $F$ and $G$, there exists a family of functions $\{(u^o_\epsilon,u^i_\epsilon)\}_{\epsilon\in]0,\epsilon'[}$, with $\epsilon'\in]0,\epsilon_0[$, such that each pair $(u^o_\epsilon,u^i_\epsilon)$ is a solution of  \eqref{princeq} for the corresponding value of $\epsilon$. Moreover, the dependence of the functions $u^o_\epsilon$ and $u^i_\epsilon$ upon the parameter $\epsilon$ can be described in terms of real analytic maps of $\epsilon$. The aim of this paper is to show that each of such solutions $(u^o_\epsilon,u^i_\epsilon)$ is locally unique. In other words, we shall verify that, for $\epsilon>0$ smaller than a certain  $\epsilon^\ast \in ]0,\epsilon_0[$, any solution  $(v^o,v^i)$ of problem \eqref{princeq} that is ``close enough'' to the pair $(u^o_\epsilon,u^i_\epsilon)$   has to coincide with $(u^o_\epsilon,u^i_\epsilon)$.  We will see that the ``distance''  from the solution $(u^o_\epsilon,u^i_\epsilon)$ can be measured solely in terms of the  $C^{1,\alpha}$ norm of the trace  of the rescaled function $v^i(\epsilon\cdot)$ on $\partial\Omega^i$. More precisely, we will prove that there is $\delta^*>0$ such that, if $\epsilon\in]0,\epsilon^*]$, $(v^o ,v^i )$ is a solution of \eqref{princeq}, and 
	\begin{equation}\label{cond}
	\left\| v^i(\epsilon \cdot ) - u^i_\epsilon(\epsilon \cdot ) \right\| _{C^{1,\alpha}(\partial \Omega^i)} < \epsilon\delta^\ast,
	\end{equation}
	then 
	\[
	(v^o,v^i)=(u^o_\epsilon,u^i_\epsilon)
	\] 
	(cf.~Theorem \ref{Thmunisol2} below). We note that in  \cite{DaMi15} it has been shown that for  a ``big'' inclusion (that is, with $\epsilon>0$ fixed)  problem \eqref{princeq}  may have  solutions that are not locally unique. Such a  different result reflects the fact that  the  solutions of \cite{DaMi15} are obtained by the Schauder fixed-point theorem, whereas the solutions of the present paper are obtained by means of the implicit function theorem.  
	
We will not provide in this paper an estimate for the values of $\epsilon^*$ and $\delta^*$. In principle, they could be obtained studying the norm of certain integral operators that we employ in our proofs. We observe that, in specific applications, for example to the heat conduction in composite materials, it might be important to understand if $\epsilon^*$ and $\delta^*$ are big enough for the model that one adopts. In particular, for very small $\epsilon^*$, that correspond to very small inclusions,  and very small $\delta^*$, that corresponds to very small differences of temperature, one may have to deal with different physical models.
	
	We also observe that uniqueness results are not new in the applications of the functional analytic approach to nonlinear boundary value problems (see, e.g.,  the above mentioned papers \cite{DaLa10,La07,La10}). However, the results so far presented concern the uniqueness of the entire family of solutions rather than the uniqueness of a single solution for $\epsilon>0$   fixed. For our specific problem \eqref{princeq}, a uniqueness result for the family $\{(u^o_\epsilon,u^i_\epsilon)\}_{\epsilon\in]0,\epsilon'[}$ would consist in proving that if    $\{(v^o_\epsilon,v^i_\epsilon)\}_{\epsilon\in]0,\epsilon'[}$ is another family of solutions which satisfies a certain limiting condition, for example that 
	\[
	\lim_{\epsilon\to 0}\epsilon^{-1}\left\| v^i_\epsilon(\epsilon \cdot ) - u^i_\epsilon(\epsilon \cdot ) \right\| _{C^{1,\alpha}(\partial \Omega^i)}=0,
	\]
	then 
	\[
	(v^o_\epsilon,v^i_\epsilon)=(u^o_\epsilon,u^i_\epsilon)
	\] 
	for $\epsilon$ small enough.
	
	One can verify that the local uniqueness of a single solution under condition \eqref{cond} implies the uniqueness of the family of solutions $\{(u^o_\epsilon,u^i_\epsilon)\}_{\epsilon\in]0,\epsilon_0[}$ in the sense described here above (cf. Corollary \ref{cor} below). From this point of view, we can say that the uniqueness result presented in this paper strengthen the uniqueness result for families which is typically obtained in the application of the functional analytic approach. 
	
	The paper is organized as follows. In Section \ref{notation} we define some of the symbols used later on. Section \ref{classresult} is a section of preliminaries, where we introduce some classical results of potential theory that we need.  In Section \ref{existenceresult} we recall some results of \cite{Mo18}  concerning the existence of a family $\{(u^o_\epsilon,u^i_\epsilon)\}_{\epsilon\in]0,\epsilon'[}$ of solutions of problem \eqref{princeq}. In Section \ref{localresult} we state and prove our main Theorem \ref{Thmunisol2}. The section consists of three subsections. In the first one we prove Theorem \ref{Thmunisol1}, which is a weaker version of Theorem \ref{Thmunisol2}  and follows directly from the Implicit Function Theorem argument used to obtain the family $\{(u^o_\epsilon,u^i_\epsilon)\}_{\epsilon\in]0,\epsilon'[}$. The statement of  Theorem \ref{Thmunisol1} is similar to that of Theorem  \ref{Thmunisol2}, but the assumptions are much stronger. In particular, together with condition \eqref{cond}, we have to require other  two conditions, namely that 
	\begin{equation}\label{cond2}
	\begin{split}
	\left\| v^o  - u^o_\epsilon \right\| _{C^{1,\alpha}(\partial \Omega^o)} < \epsilon\delta^\ast\quad\text{and}\quad
	\left\| v^o(\epsilon \cdot ) - u^o_\epsilon(\epsilon \cdot ) \right\| _{C^{1,\alpha}(\partial \Omega^i)} < \epsilon\delta^\ast,
	\end{split}
	\end{equation}
	in order to prove that $(v^o,v^i)=(u^o_\epsilon,u^i_\epsilon)$. In our main Theorem  \ref{Thmunisol2} we will see that the two conditions in \eqref{cond2} can be dropped and  \eqref{cond} is sufficient.  The proof of  Theorem  \ref{Thmunisol2} is presented in Subsection \ref{second} where we first show some results on real analytic composition operators in Schauder spaces, then we turn to introduce certain auxiliary maps $N$ and $S$, and finally we will be ready to state and prove our main theorem.  In the last Subsection \ref{family}, we see that the uniqueness of the family $\{(u^o_\epsilon,u^i_\epsilon)\}_{\epsilon\in]0,\epsilon'[}$ in the sense described here above can be obtained as a corollary of Theorem \ref{Thmunisol2}. At the end of the paper we have included an Appendix where we present some (classical) results on the product and composition in Schauder spaces.

\section{Notation}\label{notation}
We denote the norm of a real normed space $X$ by $\| \cdot \| _X$. We denote by $I_X$ the identity operator from $X$ to itself and we omit the subscript $X$ where no ambiguity can occur. For $x \in X$ and $R>0$, we denote by $B_X(x,R) \equiv \{y\in X : \|y-x\|_X < R \}$ the ball of radius $R$ centered at the point $x$. If $X=\R^d$, $d \in \N \setminus \{0,1\} $, we simply write $B_d(x,R)$. If $X$ and $Y$ are normed spaces we endow the product space $X \times Y$ with the norm defined by $\| (x,y) \|_{X \times Y} = \|x\|_X + \|y\|_Y $ for all $(x,y) \in X \times Y$, while we use the Euclidean norm for $\R^d$, $d\in\mathbb{N}\setminus\{0,1\}$. For $x \in \R^d$, $x_j$ denotes the $j$-th coordinate of $x$, $|x|$ denotes the Euclidean modulus of $x$ in $\R^d$. We denote by $\mathcal{L}(X,Y)$ the Banach space of linear and continuous map of $X$ to $Y$, equipped with its usual norm of the uniform convergence on the unit sphere of $X$. If $U$ is an open subset of $X$, and $F:U \to Y$ is a Fr\'echet-differentiable map in $U$, we denote the differential of $F$ by $dF$.
The inverse function of an invertible function $f$ is denoted by $f^{(-1)}$, while the reciprocal of a function $g$ is denoted by $g^{-1}$. Let $\Omega \subseteq \R^n$. Then $\overline{\Omega}$ denotes the closure of $\Omega$ in $\R^n$, $\partial \Omega$ denotes the boundary of $\Omega$, and $\nu_\Omega$ denotes the outward unit normal to $\partial \Omega$.
Let $\Omega$ be an open subset of $\R^n$ and $m \in \N \setminus \{0\}$. The space of $m$ times continuously differentiable real-valued function on $\Omega$ is denoted by $C^m(\Omega)$. Let $r \in \N \setminus \{0\}$, $f \in (C^m(\Omega))^r$. The $s$-th component of $f$ is denoted by $f_s$ and the gradient matrix of $f$ is denoted by $\nabla f$. Let $\eta=(\eta_1, \dots ,\eta_n) \in \N^n$ and $|\eta|=\eta_1+ \dots+\eta_n$. Then $D^\eta f \equiv \frac{\partial^{|\eta|}f}{\partial x^{\eta_1}_1, \dots , \partial x^{\eta_n}_n}$. If $r=1$, the Hessian matrix of the second-order partial derivatives of $f$ is denoted by $D^2 f$. The subspace of $C^m(\Omega)$ of those functions $f$ such that $f$ and its derivatives $D^\eta f$ of order $|\eta|\le m$ can be extended with continuity to $\overline{\Omega}$ is denoted $C^m(\overline{\Omega})$. The subspace of $C^m(\overline{\Omega})$ whose functions have $m$-the order derivatives that are H\"{o}lder continuous with exponent $\alpha \in ]0,1[$ is denoted $C^{m,\alpha}(\overline{\Omega})$.  If $f \in C^{0,\alpha}(\overline{\Omega})$, then its H\"{o}lder constant is defined as $|f : \Omega|_\alpha\equiv \mbox{sup} \left\{\frac{|f(x)-f(y)|}{|x-y|^\alpha} : x,y \in \overline{\Omega}, x \neq y \right\}$. If $\Omega$ is open and bounded, then the space $C^{m,\alpha}(\overline{\Omega})$, equipped with its usual norm $\|f\|_{C^{m,\alpha}(\overline{\Omega})} \equiv \|f\|_{C^{m}(\overline{\Omega})} + \sum_{|\eta|=m}{|D^\eta f : \Omega|_\alpha}$, is a Banach space. 
We denote by $C^{m,\alpha}_{\mathrm{loc}}(\R^n \setminus \Omega)$ the space of functions on $\R^n \setminus \Omega$ whose restriction to $\overline{U}$ belongs to $C^{m,\alpha}(\overline{U})$ for all open bounded subsets $U$ of $\R^n \setminus \Omega$. On $C^{m,\alpha}_{\mathrm{loc}}(\R^n \setminus \Omega)$ we consider the natural structure of Fr\'echet space. Finally we set
\begin{equation*}
C^{m,\alpha}_{\mathrm{harm}}(\overline{\Omega}) \equiv \{ u \in C^{m,\alpha}(\overline{\Omega}) \cap C^2(\Omega): \Delta u = 0 \text{ in } \Omega \}.
\end{equation*}

We say that a bounded open subset of $\R^n$ is of class $C^{m,\alpha}$ if it is a manifold with boundary imbedded in $\R^n$ of class $C^{m,\alpha}$. In particular, if $\Omega$ is a $C^{1,\alpha}$ subset of $\R^n$, then $\partial\Omega$ is a $C^{1,\alpha}$ sub-manifold of $\R^n$ of co-dimension $1$.  
If $M$ is a $C^{m,\alpha}$ sub-manifold of $\R^n$ of dimension $d\ge 1$, we define the space $C^{m,\alpha}(M)$ by exploiting a finite local parametrization. Namely, we take a finite open covering $\mathcal{U}_1, \dots, \mathcal{U}_k$ of $M$ and  $C^{m,\alpha}$ local parametrization maps $\gamma_l : \overline{B_{d}(0,1)} \to \overline{\mathcal{U}_l}$ with $l=1,\dots, k$ and we say that $\phi\in C^{m,\alpha}(M)$ if and only if $\phi\circ\gamma_l\in C^{m,\alpha}(\overline{B_{d}(0,1)})$ for all $l=1,\dots, k$. Then for all $\phi\in C^{m,\alpha}(M)$ we define
\[
\|\phi\|_{C^{m,\alpha}(M)} \equiv \sum_{l=1}^k \|\phi\circ\gamma_l \|_{C^{m,\alpha}(\overline{B_{d}(0,1)})}\,.
\]
One verifies that different $C^{m,\alpha}$ finite atlases define the same space  $C^{m,\alpha}(M)$ and equivalent norms on it.

We retain the standard notion for the Lebesgue spaces $L^p$, $p\ge 1$. If $\Omega$ is of class $C^{1,\alpha}$, we denote by $d\sigma$ the area element on $\partial\Omega$. If $Z$ is a subspace of $L^1(\partial \Omega)$, then we set
\begin{equation}\label{Z0}
Z_0 \equiv \left\{ f \in Z : \int_{\partial\Omega} f \,d\sigma = 0 \right\}.
\end{equation}

\section{Classical results of potential theory}\label{classresult}
In this section we present some classical results of potential theory. For the proofs we refer to Folland \cite{Fo95}, Gilbarg and Trudinger \cite{GiTr83}, Schauder \cite{Sc31}, and to the references therein.
\begin{defin}
	We denote by $S_n$ the function from $\R^n \setminus \{0\}$ to $\R$ defined by
	\[
	S_n(x) \equiv	\frac{|x|^{2-n}}{(2-n) s_n} \qquad \forall x \in \R^n \setminus \{0\}, 
	\]
	where $s_n$ denotes the $(n-1)$-dimensional measure of $\partial B_n(0,1)$.
\end{defin}    

$S_n$ is well known to be a fundamental solution of the Laplace operator.
Now let $\Omega$ be an open bounded connected subset of $\R^n$ of class $C^{1,\alpha}$.

\begin{defin}	
	We denote by $w_{\Omega}[\mu]$ the double layer potential with density $\mu$ given by
	\begin{equation*}
	w_{\Omega}[\mu](x) \equiv - \int_{\partial \Omega}{\nu_\Omega(y) \cdot \nabla S_n(x-y) \mu(y) \,d\sigma_y} \qquad \forall x \in \R^n	
	\end{equation*}
	for all $\mu \in \CO$.
	\\ 
	We denote by $W_{\partial\Omega}$ the boundary integral operator which takes $\mu \in \CO$ to the function $W_{\partial\Omega}[\mu]$ defined by
	\[
	W_{\partial\Omega}[\mu](x) \equiv - \int_{\partial \Omega}{\nu_\Omega(y) \cdot \nabla S_n(x-y) \mu(y) \,d\sigma_y} \qquad \forall x \in \partial \Omega.
	\]
\end{defin}
It is well known that, if $\mu \in \CO$, then $w_{\Omega}[\mu]_{| \Omega}$ admits a unique continuous extension to $\overline{\Omega}$, which we denote by $w^+_{\Omega}[\mu]$, and $w_{\Omega}[\mu]_{| \R \setminus \overline{\Omega}}$ admits a unique continuous extension to $\R \setminus \Omega$, which we denote by $w^-_{\Omega}[\mu]$.

In the following Theorem \ref{doublepot} we summarize classical results in potential theory.

\begin{teo}\label{doublepot}
	The following statements holds.
	\begin{enumerate}
		
		\item[(i)] Let $\mu \in \CO$. Then the function $w_{\Omega}[\mu]$ is harmonic in $\R^n\setminus \partial\Omega$ and at infinity. Moreover, we have the following jump relations
		\begin{equation*}
			w^\pm_{\Omega}[\mu](x) = \left( \pm \frac{1}{2} I + W_{\partial\Omega} \right)[\mu](x) \qquad \forall x \in \partial \Omega;
		\end{equation*}
		
		\item[(ii)] The map from $C^{1,\alpha}(\partial\Omega)$ to $C^{1,\alpha}(\overline{\Omega})$ which takes $\mu$ to $w^+_{\Omega}[\mu]$ is linear and continuous and the map from $C^{1,\alpha}(\partial\Omega)$ to $C^{1,\alpha}_{\mathrm{loc}}(\R\setminus\Omega)$ which takes $\mu$ to $w^-_{\Omega}[\mu]$ is linear and continuous;
		
		\item[(iii)] The map which takes $\mu$ to $W_{\partial\Omega}[\mu]$ is continuous from $\CO$ to itself;
		
		\item[(iv)] If $\R\setminus\overline{\Omega}$ is connected, then the operator $\frac{1}{2} I + W_{\partial\Omega}$ is an isomorphism from $\CO$ to itself.
	\end{enumerate}
\end{teo}
\section{An existence result for the solutions of problem $(\ref{princeq})$}\label{existenceresult}
In this section we recall some results of \cite{Mo18} on the existence of a family of solutions for problem \eqref{princeq}. In what follows $\mathfrak{u}^o$ denotes the unique solution in $C^{1,\alpha}(\overline{\Omega^o})$ of the interior Dirichlet problem in $\Omega^o$ with boundary datum $f^o$, namely
\[
\begin{cases}
\Delta \mathfrak{u}^o=0&\text{in }\Omega^o\,,\\
\mathfrak{u}^o=f^o&\text{on }\partial\Omega^o\,.
\end{cases}
\]

Then we have the following Proposition \ref{rappresentsol}, where harmonic functions in $\overline{\Omega(\epsilon)}$ and $\overline{\epsilon \Omega^i}$ are represented in terms of $\mathfrak{u}^o$, double layer potentials with appropriate densities, and a suitable restriction of the  fundamental solution $S_n$ (cf. \cite[Prop. 5.1]{Mo18}).

\begin{prop}\label{rappresentsol}
	Let $\epsilon \in ]0,\epsilon_0[$. The map 
	\[
	(U^o_\epsilon[\cdot,\cdot,\cdot,\cdot], U^i_\epsilon[\cdot,\cdot,\cdot,\cdot])
	\]
	from $\COo \times \COi_0 \times \R \times \COi$ to $C^{1,\alpha}_{\mathrm{harm}}(\overline{\Omega(\epsilon)}) \times \COeiharm$ which takes $(\phi^o,\phi^i,\zeta,\psi^i)$ to the pair of functions
	\begin{equation*}
	(U^o_\epsilon[\phi^o,\phi^i,\zeta,\psi^i], U^i_\epsilon[\phi^o,\phi^i,\zeta,\psi^i])
	\end{equation*}
	defined by
	\begin{equation}\label{rappsol}
	\begin{aligned}
	& U^o_\epsilon[\phi^o,\phi^i,\zeta,\psi^i](x)\\
	&\  \equiv \mathfrak{u}^o(x) + \epsilon w^+_{\Omega^o}[\phi^o](x) + \epsilon w^-_{\epsilon\Omega^i}\left[\phi^i\left(\frac{\cdot}{\epsilon}\right)\right](x) + \epsilon^{n-1}\zeta\, S_n(x)  && \qquad \forall x \in \overline{\Omega(\epsilon)}, 
	\\
	& U^i_\epsilon[\phi^o,\phi^i,\zeta,\psi^i](x) \equiv \epsilon w^+_{\epsilon\Omega^i}\left[\psi^i\left(\frac{\cdot}{\epsilon}\right)\right](x) + \zeta^i && \qquad \forall x \in \epsilon \overline{\Omega^i},
	\end{aligned}
	\end{equation}
	is bijective.
\end{prop}
We recall that $\COi_0$ is the subspace of $\COi$ consisting of the functions with zero integral mean on $\partial\Omega^i$ (cf.~\eqref{Z0}).
The following Lemma \ref{isolem} provides an isomorphism between $\COi_0 \times \R$ and $\COi$ (cf. \cite[Lemma 4.2]{Mo18}).
\begin{lemma}\label{isolem}
	The map from $\COi_0 \times \R$ to $\COi$ which takes $(\mu,\xi)$ to the function
	\begin{equation*}
	J[\mu,\xi] \equiv \left(-\frac{1}{2}I + W_{\partial\Omega^i}\right) [\mu] +\xi\,{S_n}_{|\partial \Omega^i} 
	\end{equation*}
	is an isomorphism.
\end{lemma}

Let $F$ be as in \eqref{FGf^o}. We indicate by $\partial_\epsilon F$ and $\partial_\zeta F$ the partial derivative of $F$ with respect to the first the last argument, respectively.
We shall exploit the following assumptions:
\begin{equation}\label{zetaicond}
	\begin{split}
	&\text{There exist $\zeta^i \in \R$ such that $F(0,\cdot,\zeta^i) = \mathfrak{u}^o(0)$}\\
	&\text{and $(\partial_\zeta F)(0,\cdot,\zeta^i)$ is constant and positive.}
	\end{split}
\end{equation}
and
\begin{equation}\label{F1}
	\begin{split}
		&\mbox{For each $t\in\partial\Omega^i$ fixed, the map from $]-\epsilon_0,\epsilon_0[ \times \R$ to $\R$}
		\\
		&\mbox{which takes  $(\epsilon,\zeta)$  to  $F(\epsilon,t,\zeta)$ is of class $C^2$.}  
	\end{split}
\end{equation}
Then we have the following Lemma \ref{-Taylem}, concerning the Taylor expansion of the functions $\mathfrak{u}^o$ and $F$ (cf. \cite[Lemmas 5.2, 5.3]{Mo18}).
\begin{lemma}\label{-Taylem}
	Let \eqref{zetaicond} and \eqref{F1} hold true. Let $a,b \in \R$. Then
	\begin{equation*}
	\begin{split}
	&F(\epsilon,t,a+\epsilon b)\\
	&\quad = F(0,t,a) + \epsilon (\partial_\epsilon F) (0,t,a) + \epsilon b (\partial_\zeta F) (0,t,a) + \epsilon^2 \tilde{F}(\epsilon,t,a,b),
	\end{split}
	\end{equation*}
	for all $(\epsilon,t) \in ]-\epsilon_0,\epsilon_0[ \times \partial\Omega^i$, where 
	\begin{equation*}
	\begin{split}
	&\tilde{F}(\epsilon,t,a,b) \equiv \int_{0}^{1}  (1-\tau) \{ (\partial^2_\epsilon F)(\tau\epsilon,t,a + \tau\epsilon b)\\
	&\qquad +  2b(\partial_\epsilon \partial_\zeta F)(\tau\epsilon,t,a + \tau\epsilon b)  + b^2(\partial^2_\zeta F)(\tau\epsilon,t,a + \tau\epsilon b) \} \,d\tau .
	\end{split}	
	\end{equation*}
	Moreover 
	\begin{equation*}
	\mathfrak{u}^o(\epsilon t)-F(0,t,\zeta^i)=\epsilon\, t\cdot\nabla \mathfrak{u}^o(0)+\epsilon^2\,\tilde{\mathfrak{u}}^o(\epsilon,t)
	\end{equation*}
	for all $(\epsilon,t) \in ]-\epsilon_0,\epsilon_0[ \times \partial\Omega^i$, where 
	\begin{equation*}
	\tilde{\mathfrak{u}}^o(\epsilon,t) \equiv\int_0^1(1-\tau)\sum_{i,j=1}^nt_i\,t_j\, (\partial_{x_i}\partial_{x_j} \mathfrak{u}^o)(\tau\epsilon t)\,d\tau\,.
	\end{equation*}
\end{lemma}

Then we introduce a notation for the superposition operators. 
\begin{defin}\label{Nemytskii}
	If $H$ is a function from  $]-\epsilon_0,\epsilon_0[ \times \partial \Omega ^i \times \R$ to $\R$, then we denote by $\mathcal{N}_H$ the (nonlinear nonautonomous) superposition operator which takes a pair $(\epsilon,v)$ consisting of  a real number  $\epsilon\in ]-\epsilon_0,\epsilon_0[$ and of a function $v$ from $\partial\Omega^i$ to $\R$ to the function $\mathcal{N}_H(\epsilon,v)$ defined by
	\[
	\mathcal{N}_H(\epsilon,v) (t) \equiv H(\epsilon,t,v(t))\qquad\forall t\in\partial\Omega^i\,.
	\]
\end{defin} 
Here the letter ``$\mathcal{N}$" stands for ``Nemytskii operator". Having introduced Definition \ref{Nemytskii}, we can now formulate the following assumption on the funtions $F$ and $G$ of \eqref{FGf^o}:
\begin{equation}\label{realanalhp}
\begin{split}
&\text{For all $(\epsilon,v)\in]-\epsilon_0,\epsilon_0[\times \COi$ we have $\mathcal{N}_F(\epsilon,v)\in\COi$}\\
&\text{and $\mathcal{N}_G(\epsilon,v)\in C^{0,\alpha}(\partial\Omega^i)$. Moreover, the superposition operator $\mathcal{N}_F$ is real}\\
&\text{analytic from $ ]-\epsilon_0,\epsilon_0[ \times \COi$ to $\COi$ and the superposition}\\
&\text{operator $\mathcal{N}_G$ is real analytic from $ ]-\epsilon_0,\epsilon_0[ \times \COi$ to $C^{0,\alpha}(\partial\Omega^i)$.}
\end{split}
\end{equation}
Then, for real analytic superposition operators we have the following Proposition \ref{differenzialeN_H} (cf.~ Lanza de Cristoforis \cite[Prop 5.3]{La07}).

\begin{prop}\label{differenzialeN_H}
	If $H$ is a function from $]-\epsilon_0,\epsilon_0[ \times \partial \Omega ^i \times \R$ to $\R$ such that the superposition operator $\mathcal{N}_H$ is real analytic from $ ]-\epsilon_0,\epsilon_0[ \times \COi$ to $\COi$, then  the partial differential of $\mathcal{N}_H$ with respect to the second variable $v$, computed at the point $(\epsilon,\overline{v}) \in ]-\epsilon_0,\epsilon_0[ \times C^{1,\alpha}(\partial\Omega^i)$, is the linear operator $d_v \mathcal{N}_H (\epsilon, \overline{v})$ defined by
	\begin{equation}\label{d_vN_Hformula}
	d_v \mathcal{N}_H (\epsilon, \overline{v}). \tilde{v} = \mathcal{N}_{(\partial_\zeta H)} (\epsilon,\overline{v}) \tilde{v} \qquad \forall \tilde{v} \in \COi. 
	\end{equation}
	The same result holds replacing the domain  and the target space of the operator $\mathcal{N}_H$ with $]-\epsilon_0,\epsilon_0[ \times C^{0,\alpha}(\partial\Omega^i)$ and $C^{0,\alpha}(\partial\Omega^i)$, respectively, and using functions $\overline{v},\tilde{v} \in C^{0,\alpha}(\partial\Omega^i)$ in \eqref{d_vN_Hformula}. 
\end{prop}

In what follows we will exploit an auxiliary map $M=(M_1,M_2,M_3)$ from $]-\epsilon_0,\epsilon_0[  \times \COo \times \COi_0 \times \R \times \COi$ to $\COo \times \COi \times C^{0,\alpha}(\partial\Omega^i)$ defined by 
\begin{align*}
&M_1[\epsilon, \phi^o,\phi^i, \zeta, \psi^i](x) \equiv  \left(\frac{1}{2}I + W_{\partial\Omega^o}\right)[\phi^o](x) \nonumber
\\ &\quad -  \epsilon^{n-1} \int_{\partial\Omega^i}{\nu_{\Omega^i}(y) \cdot \nabla S_n(x-\epsilon y) \phi^i(y) \,d\sigma_y} +  \epsilon^{n-2} S_n(x) \zeta\qquad\quad \forall x \in \partial\Omega^o\,, 
\\
&M_2[\epsilon, \phi^o, \phi^i, \zeta, \psi^i](t) \equiv  \, t \cdot \nabla \mathfrak{u}^o(0) + \epsilon \tilde{\mathfrak{u}}^o(\epsilon, t) + \left(-\frac{1}{2}I + W_{\partial\Omega^i}\right)[\phi^i](t)  
\\
&\quad + \zeta\, S_n(t)  + w^+_{\Omega^o}[\phi^o](\epsilon t) - (\partial_\epsilon F) (0,t,\zeta^i) - (\partial_\zeta F) (0,t,\zeta^i) \nonumber
\\
&\quad \times \left(\frac{1}{2}I + W_{\partial\Omega^i}\right)[\psi^i](t) - \epsilon \tilde{F}\left(\epsilon,t,\zeta^i,\left(\frac{1}{2}I + W_{\partial\Omega^i}\right)[\psi^i](t)\right)\ \forall t \in \partial\Omega^i,  \nonumber
\\
&M_3[\epsilon, \phi^o,\phi^i, \zeta, \psi^i](t) \equiv  \, \nu_{\Omega^i}(t) \cdot \left( \nabla \mathfrak{u}^o(\epsilon t) + \epsilon \nabla w^+_{\Omega^o}[\phi^o](\epsilon t) + \nabla w^-_{\Omega^i}[\phi^i](t)  \right.
\\ 
&\quad \left. +  \nabla S_n(t) \zeta - \nabla w^+_{\Omega^i}[\psi^i](t) \right) - G \left(\epsilon,t,\epsilon \left(\frac{1}{2}I + W_{\partial\Omega^i}\right)[\psi^i](t) + \zeta^i \right)   \nonumber
\\ 
&\quad \qquad\qquad\qquad\qquad\qquad\qquad\qquad\qquad\qquad\quad\qquad\qquad\qquad\qquad \forall t \in \partial\Omega^i\,,
\end{align*}
for all $(\epsilon,  \phi^o, \phi^i, \zeta, \psi^i) \in ]-\epsilon_0,\epsilon_0[ \times \COo \times \COi_0 \times \R \times \COi$. 

In Theorem \ref{existenceThm} here below we summarize some results of \cite{Mo18} on the operator $M$ (cf. \cite[Prop. 7.1, 7.5, 7.6, Lem. 7.7, Thm. 7.8]{Mo18}).
\begin{teo}\label{existenceThm}
	Let assumptions \eqref{zetaicond}, \eqref{F1} and \eqref{realanalhp} hold. Then the following statement holds.
	\begin{enumerate}
		\item[(i)] The map $M$ is real analytic  from $]-\epsilon_0,\epsilon_0[  \times \COo \times \COi_0 \times \R \times \COi$ to $\COo \times \COi \times C^{0,\alpha}(\partial\Omega^i)$.
		
		\item[(ii)] Let $\epsilon \in ]0,\epsilon_0[$ and $(\phi^o,\phi^i,\zeta,\psi^i) \in \COo \times \COi_0 \times \R \times \COi$. Then the pair 
		\[
		(u^o_\epsilon[\phi^o,\phi^i,\zeta,\psi^i],u^i_\epsilon[\phi^o,\phi^i,\zeta,\psi^i])
		\]
		defined by \eqref{rappsol} is a solution of \eqref{princeq} if and only if 
		\begin{equation}\label{Me=0.e1}
		M[\epsilon,  \phi^o, \phi^i, \zeta, \psi^i]=(0,0,0)\,.
		\end{equation}
		
		\item[(iii)] The equation
		\begin{equation*}
		M[0,  \phi^o, \phi^i, \zeta, \psi^i]=(0,0,0)
		\end{equation*}
		has a unique solution $(\phi^o_0,\phi^i_0,\zeta_0,\psi^i_0) \in \COo \times \COi_0 \times \R \times \COi$.
		
		\item[(iv)] The partial differential of $M$ with respect to $(\phi^o, \phi^i, \zeta, \psi^i)$ evaluated at $(0,\phi^o_0, \phi^i_0, \zeta_0, \psi^i_0)$, which
		we denote by
		\begin{equation*}\label{dMiso.eq1}
		\partial_{(\phi^o,\phi^i,\zeta,\psi^i)} M[0,\phi^o_0, \phi^i_0, \zeta_0, \psi^i_0]\,,
		\end{equation*}
		is an isomorphism from $\COo \times \COi_0 \times \R \times \COi$ to $\COo \times \COi \times C^{0,\alpha}(\partial\Omega^i)$.
		
		\item[(v)] There exist $\epsilon' \in ]0,\epsilon_0[$, an open neighborhood $U_0$ of $(\phi^o_0, \phi^i_0, \zeta_0, \psi^i_0)$ in $\COo \times \COi_0 \times \R \times \COi$, and a real analytic map 
		\begin{equation*}
		(\Phi^o,\Phi^i,Z,\Psi^i):\;]-\epsilon',\epsilon'[\to U_0
		\end{equation*}
		such that  the set of zeros of $M$ in $]-\epsilon',\epsilon'[ \times U_0$ coincides with the graph of $(\Phi^o[\cdot],\Phi^i[\cdot],Z[\cdot],\Psi^i[\cdot])$. In particular,
		\begin{equation*}
		(\Phi^o[0],\Phi^i[0],Z[0],\Psi^i[0]) = (\phi^o_0, \phi^i_0, \zeta_0, \psi^i_0).
		\end{equation*}
	\end{enumerate}
\end{teo}

Then, in view of Theorem \ref{existenceThm} (ii) and Theorem \ref{existenceThm} (v), we can  introduce a family of solutions $\{(u^o_\epsilon, u^i_\epsilon)\}_{\epsilon \in ]0,\epsilon'[}$ for problem \eqref{princeq}.

\begin{teo}\label{uesol}
	Let assumptions  \eqref{zetaicond}, \eqref{F1}, and \eqref{realanalhp} hold true. Let $\epsilon'$ and $(\Phi^o[\cdot],\Phi^i[\cdot],Z[\cdot],\Psi^i[\cdot])$ be as in Theorem \ref{existenceThm} (v). 
	For all $\epsilon\in]0,\epsilon'[$, let	
	\[
	\begin{aligned}
	& u^o_\epsilon(x) \equiv U^o_\epsilon[\Phi^o[\epsilon] , \Phi^i[\epsilon], Z[\epsilon] , \Psi^i[\epsilon]](x)&&\forall x \in \overline{\Omega(\epsilon)}\,,\\
	& u^i_\epsilon(x) \equiv U^i_\epsilon[\Phi^o[\epsilon] , \Phi^i[\epsilon], Z[\epsilon], \Psi^i[\epsilon]](x)	&&\forall x \in \epsilon\overline{\Omega^i}\,,
	\end{aligned}
	\]
	with $U^o_\epsilon[\cdot,\cdot,\cdot,\cdot]$ and $U^i_\epsilon[\cdot,\cdot,\cdot,\cdot]$  defined as in \eqref{rappsol}. Then the pair of functions $(u^o_\epsilon,u^i_\epsilon) \in C^{1,\alpha}_{\mathrm{harm}}(\overline{\Omega(\epsilon)}) \times \COeiharm$ is a solution of \eqref{princeq} for all $\epsilon \in ]0,\epsilon'[$.
\end{teo}

\section{Local uniqueness of the solution $(u^o_\epsilon,u^i_\epsilon)$}\label{localresult}

In this section we prove local uniqueness results for the family of solutions $\{(u^o_\epsilon, u^i_\epsilon)\}_{\epsilon \in ]0,\epsilon'[}$ of Theorem \ref{uesol}. We will denote by $\mathcal{B}_{0,r}$ the ball in the product space $\COo \times \COi_0 \times \R \times \COi$ of radius $r>0$ and centered in the $4$-tuple $(\phi^o_0, \phi^i_0, \zeta_0, \psi^i_0)$ of Theorem \ref{existenceThm} (iii). Namely, we set
\begin{equation}\label{B0r}
\begin{split}
&\mathcal{B}_{0,r}  \equiv \bigg\{(\phi^o, \phi^i, \zeta, \psi^i)\in \COo \times \COi_0 \times \R \times \COi\,:\\
&\ \|\phi^o-\phi^o_0\|_{\COo}+ \|\phi^i-\phi^i_0\|_{\COi}+ |\zeta-\zeta_0|+ \|\psi^i-\psi^i_0\|_{\COi}<r\bigg\}\end{split}
\end{equation}
for all $r>0$.
Then for $\epsilon'$ as in Theorem \ref{existenceThm} (v), we denote by $\Lambda=(\Lambda_1,\Lambda_2)$ the map from $]-\epsilon',\epsilon'[ \times \COo \times \COi_0 \times \R$ to $\COo \times \COi$ defined by
\begin{equation}\label{Aoperator}
\begin{aligned}
\Lambda_1[\epsilon,\phi^o,\phi^i,\zeta](x) &\equiv    \left( \frac{1}{2}I + W_{\partial\Omega^o}\right) [\phi^o](x)\\
& -  \epsilon^{n-1} \int_{\partial\Omega^i}{\nu_{\Omega^i}(y) \cdot \nabla S_n(x -\epsilon y) \phi^i(y) \,d\sigma_y} 
\\
& +  \epsilon^{n-2} S_n (x) \zeta && \qquad \forall x \in \partial \Omega^o,  
\\ 
\Lambda_2[\epsilon,\phi^o,\phi^i,\zeta](t)& \equiv 
\left(-\frac{1}{2}I + W_{\partial\Omega^i}\right) [\phi^i](t) +  w^+_{\Omega^o}[\phi^o](\epsilon t)\\
&
+ S_n(t) \zeta && \qquad\forall t \in \partial \Omega^i,  
\end{aligned}
\end{equation}
for all $(\epsilon,\phi^o,\phi^i,\zeta) \in ]-\epsilon',\epsilon'[ \times \COo \times \COi_0 \times \R$. We now prove the following.

\begin{prop}\label{propAoperator}
	There exist $\epsilon'' \in ]0,\epsilon'[$ and $C \in ]0,+\infty[$ such that the operator $\Lambda[\epsilon,\cdot,\cdot,\cdot]$ from $\COo \times \COi_0 \times \R$ to $\COo \times \COi$ is linear continuous and invertible for all $\epsilon \in ]-\epsilon'',\epsilon''[$ fixed and such that
	\begin{equation*}
	\|\Lambda[\epsilon,\cdot,\cdot,\cdot] ^ {(-1)}\|_{\mathcal{L}(\COo \times \COi,\COo \times \COi_0 \times \R)} \leq C
	\end{equation*}
	uniformly for $\epsilon \in ]-\epsilon'',\epsilon''[$.
\end{prop}

\begin{proof}
	By the mapping properties of the double layer potential (cf. Theorem \ref{doublepot} (ii) and Theorem \ref{doublepot} (iii)) and of integral operators with real analytic kernels (cf. Lanza de Cristoforis and Musolino \cite{LaMu13}) one verifies that  the map from $]-\epsilon',\epsilon'[$ to $\mathcal{L}(\COo \times \COi_0 \times \R,\COo \times \COi)$ which takes $\epsilon$ to $\Lambda[\epsilon,\cdot,\cdot,\cdot]$ is continuous. Since the set of invertible operators is open in the space $\mathcal{L}(\COo \times \COi_0 \times \R,\COo \times \COi)$, to complete the proof it suffices to show that for $\epsilon=0$  the map which takes  $(\phi^o,\phi^i,\zeta) \in \COo \times \COi_0 \times \R$ to 
	\begin{equation*}
	\begin{split}
	&\Lambda[0,\phi^o,\phi^i,\zeta]\\
	&\qquad =  
	\left(
	\left( \frac{1}{2}I + W_{\partial\Omega^o}\right)[\phi^o] ,
	\left(-\frac{1}{2}I + W_{\partial\Omega^i}\right) [\phi^i] +  w^+_{\Omega^o}[\phi^o](0)
	+ {S_n}_{|\partial\Omega^i} \zeta\right)\\
	&\qquad\in \COo \times \COi	
	\end{split}
	\end{equation*}
is invertible. To prove it, we verify that it is a bijection and then we exploit  the Open Mapping Theorem. So let $(h^o,h^i) \in \COo\times\COi$. We claim that there exists a unique $(\phi^o,\phi^i,\zeta) \in \COo \times \COi_0 \times \R$ such that 
	\begin{equation}\label{propAoperator.eq1}
	\Lambda[0,\phi^o,\phi^i,\zeta] = (h^o,h^i).
	\end{equation}
	Indeed, by Theorem \ref{doublepot} (iv), $ \frac{1}{2}I + W_{\partial\Omega^o}$ is an isomorphism from $\COo$ into itself and there exists a unique $\phi^o\in  \COo$ that satisfies the first equation of \eqref{propAoperator.eq1}. Moreover, by Lemma \ref{isolem}, the map from $\COi_0 \times \R$ to $\COi$ that takes $(\phi^i,\zeta)$ to $\left(-\frac{1}{2}I + W_{\partial\Omega^i}\right) [\phi^i] 
	+ {S_n}_{|\partial\Omega^i} \zeta,$ is an isomorphism.	Hence, there exists a unique $(\phi^i,\zeta) \in \COi_0 \times \R$ such that
	\begin{equation*}
	\left(-\frac{1}{2}I + W_{\partial\Omega^i}\right) [\phi^i] 
	+ {S_n}_{|\partial\Omega^i} \zeta = h^i - w^+_{\Omega^o}[\phi^o](0).
	\end{equation*}
	Accordingly, there exists a unique $(\phi^i,\zeta) \in \COi_0 \times \R$  that satisfies the second equation of \eqref{propAoperator.eq1}.
	Thus $\Lambda[0,\cdot,\cdot,\cdot]$ is an isomorphism from $\COo \times \COi_0 \times \R$ to $\COo \times \COi$ and the proof is complete.	
\end{proof}

\subsection{A first local uniqueness result}
We are now ready to state our first local uniqueness result for the solution $(u^o_\epsilon,u^i_\epsilon)$. Theorem \ref{Thmunisol1} here below is, in a sense,  a consequence of an argument based on the Implicit Function Theorem for real analytic maps (see, for example, Deimling \cite[Thm. 15.3]{De85}) that has been used in \cite{Mo18} to prove the existence of such solution. We shall see in the following Subsection \ref{second} that the statement of Theorem \ref{Thmunisol1} holds under much weaker assumptions. 

\begin{teo}\label{Thmunisol1}
	Let assumptions  \eqref{zetaicond}, \eqref{F1}, and \eqref{realanalhp} hold true.  Let $\epsilon' \in ]0,\epsilon_0[$ be as in Theorem \ref{existenceThm} (v).
	Let $\{(u^o_\epsilon,u^i_\epsilon)\}_{\epsilon \in ]0,\epsilon'[}$ be as in Theorem \ref{uesol}. Then there exist $\epsilon^*\in]0,\epsilon'[$ and $\delta^\ast \in ]0,+\infty[$ such that the following property holds:
	
	If $\epsilon \in ]0,\epsilon^*[$ and $(v^o,v^i) \in C^{1,\alpha}(\Omega(\epsilon)) \times C^{1,\alpha}(\epsilon\Omega^i)$ is a solution of problem \eqref{princeq} with
	\begin{align}
	\left\|v^o - u^o_\epsilon \right\| _{C^{1,\alpha}(\partial \Omega^o)} &\leq \epsilon\delta^\ast, \label{uo}
	\\
	\left\| v^o(\epsilon \cdot ) - u^o_\epsilon(\epsilon \cdot ) \right\| _{C^{1,\alpha}(\partial \Omega^i)} &\leq \epsilon\delta^\ast, \label{uoe}
	\\
	\left\| v^i(\epsilon \cdot ) - u^i_\epsilon(\epsilon \cdot ) \right\| _{C^{1,\alpha}(\partial \Omega^i)} &\leq \epsilon\delta^\ast, \label{uie}
	\end{align}
	
	then
	\begin{equation*}
	(v^o,v^i)=(u^o_\epsilon,u^i_\epsilon)\,.
	\end{equation*}
\end{teo}

\begin{proof}
	Let $U_0$ be the open neighborhood of $(\phi^o_0, \phi^i_0, \zeta_0, \psi^i_0)$ in $\COo \times \COi_0 \times \R \times \COi$ introduced in Theorem \ref{existenceThm} (v). We take $K>0$ such that 
	\[ \overline{\mathcal{B}_{0,K}}\subseteq U_0\,. 
	\]
	Since $(\Phi^o[\cdot],\Phi^i[\cdot],Z[\cdot],\Psi^i[\cdot])$ is continuous (indeed real analytic) from $]-\epsilon',\epsilon'[$ to $U_0$, then there exists ${\epsilon'_*} \in ]0,\epsilon'[$ such that
	\begin{equation}\label{(Phi^o,Phi^i,Z,Psi^i)K/2}
	(\Phi^o[\eta] , \Phi^i[\eta],Z[\eta] , \Psi^i[\eta])\in \mathcal{B}_{0,K/2} \qquad\forall \eta \in ]0,\epsilon'_*[\,.
	\end{equation}	
	Let $\epsilon''$ be as in Proposition \ref{propAoperator} and let
	\begin{equation*}
	\epsilon^\ast \equiv \min\{\epsilon'_\ast,\epsilon''\}.
	\end{equation*}
	Let $\epsilon \in ]0,\epsilon^\ast[$ be fixed and let $(v^o,v^i) \in C^{1,\alpha}(\Omega(\epsilon)) \times C^{1,\alpha}(\epsilon\Omega^i)$ be a solution of problem \eqref{princeq} that satisfies \eqref{uo}, \eqref{uoe}, and \eqref{uie} for a certain  $\delta^\ast \in ]0,+\infty[$. We show that for $\delta^\ast$ sufficiently small $(v^o,v^i)=(u^o_\epsilon,u^i_\epsilon)$. By Proposition \ref{rappresentsol}, there exists a unique quadruple $(\phi^o,\phi^i,\zeta,\psi^i) \in \COo \times \COi_0 \times \R \times \COi$ such that
	\begin{align}\label{uo[]}
	v^o &= U^o_{\epsilon}[\phi^o,\phi^i,\zeta,\psi^i]  \qquad \text{in } \overline{\Omega(\epsilon)},
	\\
	\label{ui[]}
	v^i &= U^i_{\epsilon}[\phi^o,\phi^i,\zeta,\psi^i]  \qquad \text{in } \epsilon \overline{\Omega^i}.
	\end{align}
	By \eqref{uie} and by \eqref{ui[]}, we have
	\begin{equation}\label{(u^i-u_e^i)}
	\begin{split}
	\delta^\ast &\geq \left\| \frac{v^i(\epsilon \cdot ) - u^i_\epsilon(\epsilon \cdot )}{\epsilon} \right\| _{C^{1,\alpha}(\partial \Omega^i)}\\
	& =
	\left\| \frac{U^i_\epsilon[\phi^o,\phi^i,\zeta,\psi^i](\epsilon \cdot ) - u^i_\epsilon(\epsilon \cdot )}{\epsilon} \right\| _{C^{1,\alpha}(\partial \Omega^i)} 
	\\
	&
	= \left\| \frac{\epsilon w^+_{\epsilon\Omega^i}\left[\psi^i\left(\frac{\cdot}{\epsilon}\right)\right](\epsilon \cdot) + \zeta^i -\epsilon w^+_{\epsilon\Omega^i}\left[\Psi^i[\epsilon]\left(\frac{\cdot}{\epsilon}\right)\right](\epsilon \cdot) - \zeta^i}{\epsilon} \right\| _{C^{1,\alpha}(\partial \Omega^i)} 
	\\
	&
	= \left\| w^+_{\Omega^i}[\psi^i] - w^+_{\Omega^i}[\Psi^i[\epsilon]] \right\| _{C^{1,\alpha}(\partial \Omega^i)} \,.
	\end{split}
	\end{equation}
	By the jump relations in Theorem \ref{doublepot} (i), we obtain
	\begin{equation}\label{Thmunisol1.eq1}
	\left\| \left(\frac{1}{2}I + W_{\partial\Omega^i}\right)[\psi^i] - \left(\frac{1}{2}I + W_{\partial\Omega^i}\right)[\Psi^i[\epsilon]] \right\|_{\COi} \leq \, \delta^\ast.
	\end{equation}
	By Theorem \ref{doublepot} (iv), the operator $\frac{1}{2}I + W_{\partial\Omega^i}$ is a linear isomorphism from $\COi$ to itself. Then, if we denote by $D$ the norm of its inverse, namely we set
	\[
	D \equiv \left\|\left(\frac{1}{2}I + W_{\partial\Omega^i}\right)^{(-1)} \right\|_{\mathcal{L}(\COi,\COi)}\,,
	\]
	we obtain, by \eqref{uie} and by \eqref{Thmunisol1.eq1}, that
	\begin{equation}\label{psi^i}
	\begin{split}
	&\| \psi^i - \Psi^i[\epsilon] \|_{\COi}\\
	&\quad\leq \left\|\left(\frac{1}{2}I + W_{\partial\Omega^i}\right)^{(-1)} \right\|_{\mathcal{L}(\COi,\COi)} 
	\\
	&\qquad\times\left\| \left(\frac{1}{2}I + W_{\partial\Omega^i}\right)[\psi^i] - \left(\frac{1}{2}I + W_{\partial\Omega^i}\right)[\Psi^i[\epsilon]] \right\|_{\COi}\\
  &\quad \leq  D \delta^\ast.
	\end{split}
	\end{equation}
	By \eqref{uo} and \eqref{uo[]} we have
	{\small \begin{equation}\label{(u^o - u_e^o)1}
	\begin{split}
	&\delta^\ast \geq \left\| \frac{v^o - u^o_\epsilon}{\epsilon} \right\| _{C^{1,\alpha}(\partial \Omega^o)} 
	=\left\| \frac{U^o_{\epsilon}[\phi^o,\phi^i,\zeta,\psi^i] - u^o_\epsilon)}{\epsilon} \right\| _{C^{1,\alpha}(\partial \Omega^o)} 
	\\
	&= \left\| \frac{\epsilon w^+_{\Omega^o}[\phi^o -\Phi^o[\epsilon] ] + \epsilon w^-_{\epsilon\Omega^i}\left[\phi^i\left(\frac{\cdot}{\epsilon}\right) -\Phi^i[\epsilon]\left(\frac{\cdot}{\epsilon}\right)\right] + \epsilon^{n-1}\left(\zeta - Z[\epsilon] \right)\, S_n}{\epsilon} \right\| _{C^{1,\alpha}(\partial \Omega^o)}
	\\
	&= \left\|  w^+_{\Omega^o}[\phi^o -\Phi^o[\epsilon] ] +  w^-_{\epsilon\Omega^i}\left[\phi^i\left(\frac{\cdot}{\epsilon}\right) -\Phi^i[\epsilon]\left(\frac{\cdot}{\epsilon}\right)\right] + \epsilon^{n-2}\left(\zeta - Z[\epsilon] \right)\, S_n \right\| _{C^{1,\alpha}(\partial \Omega^o)}.
	\end{split}	
	\end{equation}}
	Similarly,  \eqref{uoe} and \eqref{uo[]} yield
	{\small \begin{equation}\label{(u^o - u_e^o)2}
	\begin{split}
	&\delta^\ast \geq \left\| \frac{v^o(\epsilon \cdot) - u^o_\epsilon( \epsilon\cdot)}{\epsilon} \right\| _{C^{1,\alpha}(\partial \Omega^i)} 
	=\left\| \frac{U^o_\epsilon[\phi^o,\phi^i,\zeta,\psi^i](\epsilon \cdot) - u^o_\epsilon(\epsilon\cdot )}{\epsilon} \right\| _{C^{1,\alpha}(\partial \Omega^i)}  
	\\
	&= \left\|  w^+_{\Omega^o}[\phi^o -\Phi^o[\epsilon] ](\epsilon \cdot) + w^-_{\epsilon\Omega^i}\left[\phi\left(\frac{\cdot}{\epsilon}\right) -\Phi^i[\epsilon]\left(\frac{\cdot}{\epsilon}\right)\right](\epsilon \cdot)\right.\\
	&\left.\qquad\qquad \qquad\qquad \qquad\qquad \qquad\qquad + \epsilon^{n-2}\left(\zeta - Z[\epsilon] \right)\, S_n(\epsilon \cdot) \right\| _{C^{1,\alpha}(\partial \Omega^i)}
	\\
	&= \left\|  w^-_{\Omega^i}\left[\phi^i -\Phi^i[\epsilon]\right] + w^+_{\Omega^o}[\phi^o -\Phi^o[\epsilon] ](\epsilon \cdot) + \left(\zeta - Z[\epsilon] \right)\, S_n \right\| _{C^{1,\alpha}(\partial \Omega^i)}.
	\end{split}	
	\end{equation}}
	Then, by \eqref{(u^o - u_e^o)1} and \eqref{(u^o - u_e^o)2} and by the definition of the operator $\Lambda$ in \eqref{Aoperator}, we deduce that
	\begin{equation}\label{A[phi^o,phi^i,zeta]}
	\left\| \Lambda\left[\epsilon,\phi^o - \Phi^o[\epsilon],\phi^i - \Psi^i[\epsilon],\zeta - Z[\epsilon]\right] \right\| _{\COo\times\COi} \leq 2\delta^\ast
	\end{equation}
	(see also  the jump relations for the double layer potential in Theorem \ref{doublepot} (i)).
	Now let $C>0$ as in the statement of Proposition \ref{propAoperator}.
	Then, by the membership of $\epsilon$ in $]0,\epsilon^\ast[$, we have
	\begin{equation}\label{phi^o,phi^i,zeta}
	\begin{split}
	&\left(\phi^o - \Phi^o[\epsilon],\phi^i - \Psi^i[\epsilon],\zeta - Z[\epsilon]\right)\\
	&\qquad = \Lambda[\epsilon,\cdot,\cdot,\cdot]^{(-1)} \Lambda\left[\epsilon,\phi^o - \Phi^o[\epsilon],\phi^i - \Psi^i[\epsilon],\zeta - Z[\epsilon]\right]
	\end{split}
	\end{equation}
	and, by \eqref{A[phi^o,phi^i,zeta]} and \eqref{phi^o,phi^i,zeta}, we obtain
	\begin{equation*}
	\begin{split}
	&\|\left(\phi^o -  \Phi^o[\epsilon],  \phi^i - \Psi^i[\epsilon],\zeta - Z[\epsilon]\right)  \|_{\COo \times \COi_0 \times \R}
	\\
	& \quad\leq \|\Lambda[\epsilon,\cdot,\cdot,\cdot]^{(-1)}  \|_{\mathcal{L}(\COo \times \COi,\COo \times \COi_0 \times \R)} 
	\\
    & \quad\quad\times \left\| \Lambda\left[\epsilon,\phi^o - \Phi^o[\epsilon],\phi^i - \Psi^i[\epsilon],\zeta - Z[\epsilon]\right] \right\|_{\COo \times \COi_0 \times \R}
	\\
	&\quad\leq 2 C \delta^\ast\,.
	\end{split}
	\end{equation*}
	The latter inequality, combined with \eqref{psi^i}, yields
	\begin{equation}\label{phi^o,phi^i,zeta,psi^i}
	\begin{split}
	&\|\left(\phi^o -  \Phi^o[\epsilon],  \phi^i - \Psi^i[\epsilon],\zeta - Z[\epsilon], \psi^i - \Psi[\epsilon]\right)  \|_{\COo \times \COi_0 \times \R \times \COi}\\
	 &\qquad\qquad\qquad\qquad\qquad\qquad\qquad\qquad\qquad\qquad\leq (2C + D) \delta^* \,. 
	\end{split}
	\end{equation}
	Hence, by   \eqref{(Phi^o,Phi^i,Z,Psi^i)K/2} and \eqref{phi^o,phi^i,zeta,psi^i} and by a standard computation based on the triangular inequality one sees that
	\begin{equation*}
	\begin{split}
	&\|(\phi^o, \phi^i, \zeta, \psi^i) - (\phi^o_0, \phi^i_0, \zeta_0, \psi^i_0)\|_{\COo \times \COi_0 \times \R \times \COi}\\
	&\qquad\qquad\qquad\qquad\qquad\qquad\qquad\qquad\qquad\qquad \leq (2C+D)\delta^* + \frac{K}{2}.
	\end{split}
	\end{equation*}
	Accordingly, in order to have $(\phi^o, \phi^i, \zeta, \psi^i) \in \mathcal{B}_{0,K}$, it suffices to take 
	\begin{equation*}
	\delta^\ast < \frac{K}{2(2C+D)} 
	\end{equation*} 
	in inequalities \eqref{uo}, \eqref{uoe}, and \eqref{uie}. Then, by the inclusion $\overline{\mathcal{B}_{0,K}}\subseteq U_0$ and by Theorem \ref{existenceThm} (v), we deduce that for such choice of $\delta^\ast$ we have
	\begin{equation*}
	(\phi^o, \phi^i, \zeta, \psi^i) = \left(\Phi^o[\epsilon],  \Psi^i[\epsilon],Z[\epsilon],\Psi[\epsilon]\right) 
	\end{equation*}
	and thus  $(v^o,v^i)=(u^o_\epsilon,u^i_\epsilon)$ (cf. Theorem \ref{uesol}).
\end{proof}

\subsection{A stronger local uniqueness result}\label{second}

In this Subsection we will see that we can weaken the assumptions of Theorem \ref{Thmunisol1}. In particular, we will prove in Theorem \ref{Thmunisol2} that the local uniqueness of the solution can be achieved with only one condition on the trace of the function $v^i(\epsilon\cdot)$ on $\partial\Omega^i$, instead of the three conditions used in Theorem \ref{Thmunisol1}. To prove Theorem \ref{Thmunisol2} we shall need some preliminary technical results on composition operators.

\subsubsection{Some preliminary results on composition operators}

We begin with the following Lemma \ref{lemmaboundA(ep,dot,dot)}.

\begin{lemma}\label{lemmaboundA(ep,dot,dot)}
	Let $A$ be a function from $]-\epsilon_0,\epsilon_0[ \times \overline{B_{n-1}(0,1)} \times \R $ to $\R$. Let $\mathcal{M}_A$ be the map which takes a pair $(\epsilon,\zeta)\in]-\epsilon_0,\epsilon_0[ \times \R$  to the function $\mathcal{M}_A (\epsilon,\zeta)$ defined by
	\begin{equation}\label{M_A}
	\mathcal{M}_A (\epsilon,\zeta)(z) \equiv A(\epsilon,z,\zeta)\qquad\forall z\in  \overline{B_{n-1}(0,1)}.
	\end{equation}
Let $m\in\{0,1\}$. If $\mathcal{M}_A (\epsilon,\zeta)\in C^{m,\alpha}( \overline{B_{n-1}(0,1)} )$ for all $(\epsilon,\zeta)\in]-\epsilon_0,\epsilon_0[ \times \R$ and if the map $\mathcal{M}_A$ is real analytic from $]-\epsilon_0,\epsilon_0[ \times \R$ to $C^{m,\alpha}( \overline{B_{n-1}(0,1)} )$, then for every open bounded interval $\mathcal{J}$ of $\R$ and every compact subset $\mathcal{E}$ of $]-\epsilon_0,\epsilon_0[$ there exists $C>0$ such that
	\begin{equation}\label{A<C}
	\sup_{\epsilon \in \mathcal{E}} \|A(\epsilon,\cdot,\cdot)\|_{C^{m,\alpha}( \overline{B_{n-1}(0,1)} \times \overline{\mathcal{J}})} \leq C.
	\end{equation}
\end{lemma}

\begin{proof}
We first prove the statement of Lemma \ref{lemmaboundA(ep,dot,dot)} for $m=0$. If $\mathcal{M}_A$ is real analytic from $]-\epsilon_0,\epsilon_0[ \times \R$ to $C^{0,\alpha}( \overline{B_{n-1}(0,1)} )$, then  for every $(\tilde{\epsilon},\tilde{\zeta}) \in ]-\epsilon_0,\epsilon_0[ \times \R$ there exist $M \in ]0,+\infty[$, $\rho \in ]0,1[$, and a family of coefficients $\{a_{jk}\}_{j,k \in \N} \subset C^{0,\alpha}( \overline{B_{n-1}(0,1)} )$ such that
	\begin{equation}\label{realanalcoeffajk}
	\|a_{jk}\|_{C^{0,\alpha}( \overline{B_{n-1}(0,1)} ) } \leq M \left( \frac{1}{\rho} \right)^{k+j}\qquad\forall j,k\in\mathbb{N},
	\end{equation}
	and
	\begin{equation}\label{realanalexpA}
	\mathcal{M}_A (\epsilon,\zeta)(\cdot) = \sum_{j,k = 0}^{\infty} a_{jk}(\cdot) (\epsilon-\tilde{\epsilon})^k (\zeta-\tilde{\zeta})^j\qquad\forall(\epsilon,\zeta) \in ]\tilde{\epsilon}-\rho, \tilde{\epsilon}+\rho[ \times ]\tilde{\zeta}-\rho,\tilde{\zeta}+\rho[\,,
	\end{equation} 
	where $\rho$ is less than or equal  to the radius of convergence of the series in \eqref{realanalexpA}. Now let $\mathcal{J}\subset\R$ be open and  bounded and $\mathcal{E} \subset \, ]-\epsilon_0,\epsilon_0[$ be compact. Since the product $\overline{\mathcal{J}} \times \mathcal{E}$ is compact,  a standard finite covering argument shows that   in order to prove \eqref{A<C} for $m=0$ it suffices   to find a uniform upper bound (independent of $\tilde\epsilon$ and $\tilde\zeta$)  for the quantity
\begin{equation*}
\sup_{\epsilon \in [\tilde{\epsilon} - \frac{\rho}{2},\tilde{\epsilon} + \frac{\rho}{2}]} \|A(\epsilon,\cdot,\cdot)\|_{C^{0,\alpha}( \overline{B_{n-1}(0,1)} \times [\tilde{\zeta} - \frac{\rho}{4}, \tilde{\zeta} + \frac{\rho}{4}])}.
\end{equation*}
By \eqref{M_A}, \eqref{realanalcoeffajk}, and \eqref{realanalexpA} we have
	\begin{equation}\label{C0normA}
	\begin{split}
	\sup_{\epsilon \in [\tilde{\epsilon} - \frac{\rho}{2},\tilde{\epsilon} + \frac{\rho}{2}]}& \|A(\epsilon,\cdot,\cdot)\|_{C^0(\overline{B_{n-1}(0,1)}\times [\tilde{\zeta} - \frac{\rho}{4}, \tilde{\zeta} + \frac{\rho}{4}])} 
	\\
	& \leq \sup_{\epsilon \in [\tilde{\epsilon} - \frac{\rho}{2},\tilde{\epsilon} + \frac{\rho}{2}]} \sum_{j,k=0}^{\infty} \|a_{jk}(\cdot)  (\epsilon-\tilde{\epsilon})^k (\cdot-\tilde{\zeta})^j \|_{C^0(\overline{B_{n-1}(0,1)}\times [\tilde{\zeta} - \frac{\rho}{4}, \tilde{\zeta} + \frac{\rho}{4}])} 
	\\
	&\leq \sum_{j,k=0}^{\infty} M \left(\frac{1}{\rho}\right)^{j+k} \left(\frac{\rho}{2}\right)^k \left(\frac{\rho}{2}\right)^j =  	\sum_{j=0}^{\infty} M \left(\frac{1}{2}\right)^{j+k} = 4M
	\end{split}
	\end{equation}
for all $l=1,\dots, m$. Then inequality \eqref{C0normA} yields an estimate of the $C^0$ norm of $A$. To complete the proof of \eqref{A<C} for $m=0$ we have now to study  the H\"older constant of $A(\epsilon,\cdot,\cdot)$ on $\overline{B_{n-1}(0,1)}\times [\tilde{\zeta} - \frac{\rho}{4}, \tilde{\zeta} + \frac{\rho}{4}]$. To do so,  we take $z',z'' \in \overline{B_{n-1}(0,1)}$, $\zeta',\zeta'' \in [\tilde{\zeta}-\frac{\rho}{4},\tilde{\zeta}+\frac{\rho}{4}]$, and $\epsilon \in [\tilde{\epsilon} - \frac{\rho}{2}, \tilde{\epsilon} + \frac{\rho}{2}]$, and we consider the difference
\begin{equation}\label{a'-a''}
|a_{jk}(z') (\epsilon-\tilde{\epsilon})^k (\zeta'-\tilde{\zeta})^j - a_{jk}(z'') (\epsilon-\tilde{\epsilon})^k (\zeta''-\tilde{\zeta})^j | . 
\end{equation}	
For $j \geq 1$ and $k \geq 0$ we argue as follow: we add and subtract the term $a_{jk}(z'') (\epsilon-\tilde{\epsilon})^k (\zeta'-\tilde{\zeta})^j $ inside the absolute value in  \eqref{a'-a''}, we use the triangular inequality to split the difference in two terms and then  we exploit the membership of  $a_{jk}$ in $C^{0,\alpha}( \overline{B_{n-1}(0,1)} )$ and an argument based on the Taylor expansion at the first order for the function from $[\tilde{\zeta}-\frac{\rho}{4},\tilde{\zeta}+\frac{\rho}{4}]$ to $\R$ that takes $\zeta$ to $(\zeta-\tilde{\zeta})^j$. Doing so we show that \eqref{a'-a''} is less than or equal to 
\begin{equation}\label{a'-a''.1}
	\begin{split}
	&|a_{jk}(z')-a_{jk}(z'')| |\epsilon-\tilde{\epsilon}|^k |\zeta'-\tilde{\zeta}|^j + |a_{jk}(z'')| |\epsilon-\tilde{\epsilon}|^k |(\zeta'-\tilde{\zeta})^j - (\zeta''-\tilde{\zeta})^j|
	\\
	&\quad\leq
	\|a_{jk}\|_{C^{0,\alpha}( \overline{B_{n-1}(0,1)} )} |z'-z''|^\alpha |\epsilon-\tilde{\epsilon}|^k  |\zeta'-\tilde{\zeta}|^j\\
	&\qquad +  \|a_{jk}\|_{C^{0,\alpha}( \overline{B_{n-1}(0,1)} )} |\epsilon-\tilde{\epsilon}|^k \left(j |\overline{\zeta}-\tilde{\zeta}|^{j-1} |\zeta'-\zeta''|\right)
	\end{split}
	\end{equation}
	for a suitable $\overline{\zeta} \in [\tilde{\zeta}-\frac{\rho}{4},\tilde{\zeta}+\frac{\rho}{4}]$. Then by \eqref{realanalcoeffajk}, by inequalities $|\epsilon-\tilde{\epsilon}|\le \frac{\rho}{2}$, $|\zeta'-\tilde\zeta|\le \frac{\rho}{4}$, and $|\overline\zeta-\tilde\zeta|\le \frac{\rho}{4}$, and by a straightforward computation  we see that the right hand side of \eqref{a'-a''.1} is less than or equal to 
\begin{equation}\label{a'-a''.2}
\begin{split}
	&M \left(\frac{1}{\rho}\right)^{j+k} |z'-z''|^\alpha \left(\frac{\rho}{2}\right)^k \left(\frac{\rho}{4}\right)^j + M \left(\frac{1}{\rho}\right)^{j+k} \left(\frac{\rho}{2}\right)^k j \left(\frac{\rho}{4}\right)^{j-1}  |\zeta'-\zeta''|
		\\
	&\quad = 	M \left(\frac{1}{2}\right)^j\left(\frac{1}{2}\right)^{j+k} |z'-z''|^\alpha + 4 M \rho^{-1}\frac{j}{2^j}\left(\frac{1}{2}\right)^{j+k}
	|\zeta'-\zeta''|^{1-\alpha} |\zeta'-\zeta''|^\alpha.
	\end{split}
\end{equation}
Now, since  $\zeta'$ and $\zeta''$ are taken in the interval $[\tilde{\zeta}-\frac{\rho}{4},\tilde{\zeta}+\frac{\rho}{4}]$ we have $|\zeta'-\zeta''|^{1-\alpha}\le \left({\rho}/{2}\right)^{1-\alpha}$ and since $\rho \in ]0,1[$ and $\alpha \in ]0,1[$, we deduce that $|\zeta'-\zeta''|^{1-\alpha}\le 1$.  Moreover,  since $j\ge 1$, we have ${j}/{2^j} \leq 1$ and $(1/2)^j<1$. It follows that the right hand side of \eqref{a'-a''.2} is less than or equal to 
	\begin{equation}\label{a'-a''.3}
	\begin{split}
	&
	M 
	\left(\frac{1}{2}\right)^{j+k} |z'-z''|^\alpha 
	+ 4 M \rho^{-1}
	\left(\frac{1}{2}\right)^{j+k}
	|\zeta'-\zeta''|^\alpha\\
	&\qquad\qquad \leq
	4 M \rho^{-1}
	\left(\frac{1}{2}\right)^{j+k} \left(|z'-z''|^\alpha + |\zeta'-\zeta''|^\alpha\right)
	\end{split}
	\end{equation}
	(also note  that $\rho^{-1}>1$).
	Finally, by inequality 
	\[
	a^\alpha+b^\alpha\le 2^{1-\frac{\alpha}{2}}(a^2+b^2)^{\frac{\alpha}{2}}\,,
	\]
	which holds for all $a,b>0$, we deduce that the right hand side of \eqref{a'-a''.3} is less than or equal to 
	\begin{equation}\label{a'-a''.4}
	2^{3-\frac{\alpha}{2}} M \rho^{-1}
	\left(\frac{1}{2}\right)^{j+k} |(z',\zeta')-(z'',\zeta'')|^\alpha,
	\end{equation}
where $|(z',\zeta')-(z'',\zeta'')|$ denotes the Euclidean norm of the vector $(z',\zeta')-(z'',\zeta'')$ in $\R^{n-1}\times\R=\R^n$. Then, by \eqref{a'-a''.1}--\eqref{a'-a''.4}, we obtain that
\begin{equation}\label{a'-a''.5}
\begin{split}
&|a_{jk}(z') (\epsilon-\tilde{\epsilon})^k (\zeta'-\tilde{\zeta})^j - a_{jk}(z'') (\epsilon-\tilde{\epsilon})^k (\zeta''-\tilde{\zeta})^j |
\\
&\qquad\qquad
\le 2^{3-\frac{\alpha}{2}} M \rho^{-1}
	\left(\frac{1}{2}\right)^{j+k} |(z',\zeta')-(z'',\zeta'')|^\alpha
\end{split}
\end{equation}
for all $j \geq 1$, $k \geq 0$, and $\epsilon\in [\tilde{\epsilon} - \frac{\rho}{2}, \tilde{\epsilon} + \frac{\rho}{2}]$. Now, for every $\epsilon\in [\tilde{\epsilon} - \frac{\rho}{2}, \tilde{\epsilon} + \frac{\rho}{2}]$ we denote by $\tilde{a}_{jk, \epsilon}$ the function
\begin{equation}\label{tildeajk}
\begin{aligned}
\tilde{a}_{jk,\epsilon}\,:\; \overline{B_{n-1}(0,1)} \times \left[\tilde{\zeta} - \frac{\rho}{4},\tilde{\zeta} + \frac{\rho}{4}\right]&\to\R
\\
(z,\zeta)&\mapsto \tilde{a}_{jk,\epsilon}(z,\zeta) \equiv a_{jk}(z) (\epsilon-\tilde{\epsilon})^k (\zeta-\tilde{\zeta})^j\,.
\end{aligned}
\end{equation}
Then inequality \eqref{a'-a''.5} readily implies that
\[
\begin{split}
	&\left|\tilde{a}_{jk,\epsilon}: \overline{B_{n-1}(0,1)} \times \left[\tilde{\zeta} - \frac{\rho}{4},\tilde{\zeta} + \frac{\rho}{4}\right] \right|_\alpha \leq  2^{3-\frac{\alpha}{2}} M \rho^{-1}
	\left(\frac{1}{2}\right)^{j+k}\\
	 &\qquad\qquad\qquad\qquad\qquad\qquad\qquad \forall j \geq 1\,,\; k \geq 0\,,\; \epsilon\in \left[\tilde{\epsilon} - \frac{\rho}{2}, \tilde{\epsilon} + \frac{\rho}{2}\right]\,,
\end{split}
\]
	which in turn implies that
	\begin{equation}\label{lemmaboundA(ep,dot,dot)inequ4}
	\begin{split}
	& \sup_{\epsilon \in [\tilde{\epsilon} - \frac{\rho}{2},\tilde{\epsilon} + \frac{\rho}{2}]} \sum_{j=1,k=0}^{\infty} \left|\tilde{a}_{jk,\epsilon} : \overline{B_{n-1}(0,1)}\times \left[\tilde{\zeta} - \frac{\rho}{4}, \tilde{\zeta} + \frac{\rho}{4}\right] \right|_\alpha \\
	&\qquad\qquad\qquad
	\leq  \sum_{j=1,k=0}^{\infty}  2^{3-\frac{\alpha}{2}} M \rho^{-1}
	\left(\frac{1}{2}\right)^{j+k} \leq 2^{4-\frac{\alpha}{2}} M \rho^{-1}.
	\end{split}
	\end{equation}
	
	We now turn to consider \eqref{a'-a''} in the case where $j=0$ and $k \geq 0$. In such case, one verifies that the quantity in \eqref{a'-a''} is less than or equal to
\[
	   \|a_{0k}\|_{C^{0,\alpha}( \overline{B_{n-1}(0,1)} )}  |\epsilon-\tilde{\epsilon}|^k |z'-z''|^\alpha\,,
\]
	which, by \eqref{realanalcoeffajk} and by inequality $|\epsilon-\tilde{\epsilon}|\le \frac{\rho}{2}$, is less than or equal to 
	\[
  M \left( \frac{1}{\rho} \right)^{k}  \left(\frac{\rho}{2}\right)^k |z'-z''|^\alpha=M\left(\frac{1}{2}\right)^k|z'-z''|^\alpha\,.
	\]
	 Hence, for $\tilde{a}_{0k,\epsilon}$  defined as in \eqref{tildeajk} (with $j=0$) we have
	\begin{equation*}
	\begin{split}
	&\left|\tilde{a}_{0k,\epsilon} : \overline{B_{n-1}(0,1)} \times \left[\tilde{\zeta} - \frac{\rho}{4},\tilde{\zeta} + \frac{\rho}{4}\right] \right|_\alpha \leq M\left(\frac{1}{2}\right)^k\\
	&\qquad\qquad\qquad	\qquad\qquad\qquad	 \qquad \forall  k \geq 0\,,\; \epsilon\in \left[\tilde{\epsilon} - \frac{\rho}{2}, \tilde{\epsilon} + \frac{\rho}{2}\right]\,,
	\end{split}
	\end{equation*}
	which implies that
	\begin{equation}\label{lemmaboundA(ep,dot,dot)inequ3}
	\begin{split}
	&\sup_{\epsilon \in [\tilde{\epsilon} - \frac{\rho}{2},\tilde{\epsilon} + \frac{\rho}{2}]} \sum_{k=0}^{\infty} \left|\tilde{a}_{0k,\epsilon} : \overline{B_{n-1}(0,1)}\times \left[\tilde{\zeta} + \frac{\rho}{4}, \tilde{\zeta} + \frac{\rho}{4}\right] \right|_\alpha 	\\
	&\qquad\qquad\qquad\qquad\qquad\qquad\qquad\qquad
	\leq \sum_{k=0}^{\infty} M \left(\frac{1}{2}\right)^k = 2M\,.
	\end{split}
	\end{equation}
 
	Finally, by \eqref{realanalexpA},  \eqref{C0normA}, \eqref{lemmaboundA(ep,dot,dot)inequ4}, and \eqref{lemmaboundA(ep,dot,dot)inequ3} we obtain
\[
\begin{split}
&\sup_{\epsilon \in [\tilde{\epsilon} - \frac{\rho}{2},\tilde{\epsilon} + \frac{\rho}{2}]} \|A(\epsilon,\cdot,\cdot)\|_{C^{0,\alpha}( \overline{B_{n-1}(0,1)} \times [\tilde{\zeta} + \frac{\rho}{4}, \tilde{\zeta} + \frac{\rho}{4}])}
\\
&\qquad = 
\sup_{\epsilon \in [\tilde{\epsilon} - \frac{\rho}{2},\tilde{\epsilon} + \frac{\rho}{2}]} \|A(\epsilon,\cdot,\cdot)\|_{C^0(\overline{B_{n-1}(0,1)}\times [\tilde{\zeta} + \frac{\rho}{4}, \tilde{\zeta} + \frac{\rho}{4}])} \\
&\qquad\quad	
+ \sup_{\epsilon \in [\tilde{\epsilon} - \frac{\rho}{2},\tilde{\epsilon} + \frac{\rho}{2}]} 	\left|A(\epsilon,\cdot,\cdot) : \overline{B_{n-1}(0,1)} \times \left[\tilde{\zeta} - \frac{\rho}{4},\tilde{\zeta} + \frac{\rho}{4}\right] \right|_\alpha
\\
& \qquad
\le 4M
+\sup_{\epsilon \in [\tilde{\epsilon} - \frac{\rho}{2},\tilde{\epsilon} + \frac{\rho}{2}]}\sum_{j,k=0}^\infty\left|\tilde{a}_{jk,\epsilon} : \overline{B_{n-1}(0,1)} \times \left[\tilde{\zeta} - \frac{\rho}{4},\tilde{\zeta} + \frac{\rho}{4}\right] \right|_\alpha
\\
& \qquad
=4M
+\sup_{\epsilon \in [\tilde{\epsilon} - \frac{\rho}{2},\tilde{\epsilon} + \frac{\rho}{2}]}\sum_{j=1,k=0}^\infty\left|\tilde{a}_{jk,\epsilon} : \overline{B_{n-1}(0,1)} \times \left[\tilde{\zeta} - \frac{\rho}{4},\tilde{\zeta} + \frac{\rho}{4}\right] \right|_\alpha 
\\
&\qquad\quad
+\sup_{\epsilon \in [\tilde{\epsilon} - \frac{\rho}{2},\tilde{\epsilon} + \frac{\rho}{2}]}\sum_{k=0}^\infty\left|\tilde{a}_{0k,\epsilon} : \overline{B_{n-1}(0,1)} \times \left[\tilde{\zeta} - \frac{\rho}{4},\tilde{\zeta} + \frac{\rho}{4}\right] \right|_\alpha\\
&\qquad 
\leq 4M +2^{4-\frac{\alpha}{2}} M \rho^{-1}  + 2M\,. 
\end{split}	
\]
We deduce that \eqref{A<C} for $m=0$ holds with $C=6M+2^{4-\frac{\alpha}{2}} M \rho^{-1}$.

We now assume that $\mathcal{M}_A$ is real analytic from $]-\epsilon_0,\epsilon_0[ \times \R$ to the space $C^{1,\alpha}(\overline{B_{n-1}(0,1)})$ and we prove \eqref{A<C} for $m=1$. To do so we will exploit the (just proved) statement of Lemma \ref{lemmaboundA(ep,dot,dot)} for $m=0$. We begin by observing that, since the imbedding of $C^{1,\alpha}(\overline{B_{n-1}(0,1)})$ into $C^{0,\alpha}(\overline{B_{n-1}(0,1)})$ is linear and continuous, the map $\mathcal{M}_A$ is real analytic from $]-\epsilon_0,\epsilon_0[ \times \R$ to $C^{0,\alpha}(\overline{B_{n-1}(0,1)})$. Hence, by Lemma \ref{lemmaboundA(ep,dot,dot)} for $m=0$ and by the continuity of the imbedding of  $C^{0,\alpha}(\overline{B_{n-1}(0,1)} \times \overline{\mathcal{J}})$ into $C^{0}(\overline{B_{n-1}(0,1)} \times \overline{\mathcal{J}})$  we deduce that 
\begin{equation}\label{A<C1}
\sup_{\epsilon \in \mathcal{E}} \|A(\epsilon,\cdot,\cdot)\|_{C^{0}(\overline{B_{n-1}(0,1)} \times \overline{\mathcal{J}})} \leq C_1\,.
\end{equation}
Moreover, since differentials of real analytic maps are real analytic, we have that the map $\mathcal{M}_{\partial_\zeta A}= \partial_\zeta \mathcal{M}_{A}$ which takes $(\epsilon,\zeta)$ to $ \partial_\zeta A(\epsilon,\cdot,\zeta)$ is real analytic from $]-\epsilon_0,\epsilon_0[ \times \R$ to $C^{1,\alpha}(\overline{B_{n-1}(0,1)})$, and thus from $]-\epsilon_0,\epsilon_0[ \times \R$ to $C^{0,\alpha}(\overline{B_{n-1}(0,1)})$. By Lemma \ref{lemmaboundA(ep,dot,dot)} for $m=0$ it follows that 
\begin{equation}\label{dzetaA<C2}
\sup_{\epsilon \in \mathcal{E}} \| \partial_\zeta A(\epsilon,\cdot,\cdot)\|_{C^{0,\alpha}(\overline{B_{n-1}(0,1)} \times \overline{\mathcal{J}})} \leq C_2,
\end{equation}
for some $C_2>0$. Finally, we observe that the map $\partial_z$ from $C^{1,\alpha}(\overline{B_{n-1}(0,1)}) $ to $C^{0,\alpha}(\overline{B_{n-1}(0,1)}) $ that takes a function $f$ to $\partial_z f$ is linear and continuous. Then, the map $\mathcal{M}_{\partial_z A}$ which takes  $(\epsilon,\zeta)$ to the function 
\[
\partial_z A(\epsilon,z,\zeta)\qquad\forall z\in\overline{B_{n-1}(0,1)}
\]
is the composition of $\mathcal{M}_A$ and $\partial_z$. Namely, we can write 
\[
\mathcal{M}_{\partial_z A}= \partial_z\circ \mathcal{M}_{A}\,.
\] 
Since $\mathcal{M}_A$ is real analytic from $]-\epsilon_0,\epsilon_0[ \times \R$ to $C^{1,\alpha}(\overline{B_{n-1}(0,1)})$, it follows that $\mathcal{M}_{\partial_z A}$ is real analytic from $]-\epsilon_0,\epsilon_0[ \times \R$ to $C^{0,\alpha}(\overline{B_{n-1}(0,1)})$. Hence Lemma \ref{lemmaboundA(ep,dot,dot)} for $m=0$ implies that there exists $C_3>0$ such  that
\begin{equation}\label{dzA<C3}
\sup_{\epsilon \in \mathcal{E}} \| \partial_z A(\epsilon,\cdot,\cdot)\|_{C^{0,\alpha}(\overline{B_{n-1}(0,1)} \times \overline{\mathcal{J}})} \leq C_3\,.
\end{equation}
Now, the validity of \eqref{A<C} for $m=1$ is a consequence of \eqref{A<C1}, \eqref{dzetaA<C2}, and \eqref{dzA<C3}.
\end{proof}

In the sequel we will exploit Schauder spaces over suitable subsets of  $\partial\Omega^i \times\R$. We observe indeed that for all open bounded intervals $\mathcal{J}$ of $\R$, the product  $\partial\Omega^i \times\overline{\mathcal{J}}$ is a compact sub-manifold (with boundary) of co-dimension $1$ in $\R^n\times\R=\R^{n+1}$ and accordingly,  we can define the spaces $C^{0,\alpha}(\partial\Omega^i\times\overline{\mathcal{J}})$ and  $C^{1,\alpha}(\partial\Omega^i\times\overline{\mathcal{J}})$ by exploiting a finite atlas (see Section \ref{notation}).

\begin{lemma}\label{lemmaboundB(ep,dot,dot)}
	Let $B$ be a function from $]-\epsilon_0,\epsilon_0[ \times \partial\Omega^i \times \R$ to $\R$. Let $\tilde{\mathcal{N}}_B$ be the map which takes a pair $(\epsilon,\zeta)\in]-\epsilon_0,\epsilon_0[ \times \R$  to the function $\tilde{\mathcal{N}}_B (\epsilon,\zeta)$ defined by
	\[
	\tilde{\mathcal{N}}_B(\epsilon,\zeta)(t) \equiv B(\epsilon,t,\zeta)\qquad\forall t\in  \partial\Omega^i\,.
	\]
Let $m\in\{0,1\}$. If $\tilde{\mathcal{N}}_B (\epsilon,\zeta)\in C^{m,\alpha}(\partial\Omega^i)$ for all $(\epsilon,\zeta)\in]-\epsilon_0,\epsilon_0[ \times \R$ and the map $\tilde{\mathcal{N}}_B$ is real analytic from $]-\epsilon_0,\epsilon_0[ \times \R$ to $C^{m,\alpha}(\partial\Omega^i)$, then for every open bounded interval $\mathcal{J}$ of $\R$ and every compact subset $\mathcal{E}$ of $]-\epsilon_0,\epsilon_0[$ there exists $C>0$ such that
	\begin{equation}\label{lemmaboundB(ep,dot,dot).eq1}
	\sup_{\epsilon \in \mathcal{E}} \|B(\epsilon,\cdot,\cdot)\|_{C^{m,\alpha}( \partial\Omega^i \times \overline{\mathcal{J}})} \leq C\,.
	\end{equation}
\end{lemma}

\begin{proof} Since $\partial\Omega^i$ is a compact sub-manifold of class $C^{1,\alpha}$ in $\R^n$, there exist a  finite open covering $\mathcal{U}_1$, \dots, $\mathcal{U}_k$ of $\partial\Omega^i$ and  $C^{1,\alpha}$ local parametrization maps
$\gamma_l \colon \overline{B_{n-1}(0,1)} \to \overline{\mathcal{U}_l}$ with $l=1,\dots, k$. Moreover,  we can assume  without loss of generality that the norm of $C^{m,\alpha}( \partial\Omega^i)$ is defined on the atlas $\{(\overline{\mathcal{U}_l},\gamma^{(-1)}_l)\}_{l = 1, \dots ,k}$ and the norm of $C^{m,\alpha}( \partial\Omega^i\times \overline{\mathcal{J}})$  is defined on the atlas $\{(\overline{\mathcal{U}_l}\times\overline{\mathcal{J}},(\gamma^{(-1)}_l,\mathrm{id}_{\overline{\mathcal{J}}}))\}_{l = 1, \dots ,k}$, where $\mathrm{id}_{\overline{\mathcal{J}}}$ is the identity map from $\overline{\mathcal{J}}$ to itself. Then, in order to prove \eqref{lemmaboundB(ep,dot,dot).eq1} it suffices to show that
\begin{equation}\label{lemmaboundB(ep,dot,dot).eq2}
	\sup_{\epsilon \in \mathcal{E}} \|B(\epsilon,\gamma_l(\cdot),\cdot)\|_{C^{m,\alpha}(\overline{B_{n-1}(0,1)} \times \overline{\mathcal{J}})} \leq C\qquad\forall l\in\{1,\dots,k\}
\end{equation}
for some $C>0$.  Let $l\in\{1,\dots,k\}$ and let $A$ be the map from $]-\epsilon_0,\epsilon_0[ \times \overline{B_{n-1}(0,1)} \times \R $ to $\R$ defined by
\begin{equation}\label{lemmaboundB(ep,dot,dot).eq3}
A(\epsilon,z,\zeta)=B(\epsilon,\gamma_l(z),\zeta)\qquad\forall (\epsilon,z,\zeta)\in]-\epsilon_0,\epsilon_0[ \times \overline{B_{n-1}(0,1)} \times \R\,.
\end{equation}
Then, with the notation of Lemma \ref{lemmaboundA(ep,dot,dot)}, we have
\[
\mathcal{M}_A(\epsilon,\zeta)=\gamma_l^*\left(\tilde{\mathcal{N}}_B(\epsilon,\zeta)_{|\overline{\mathcal{U}_l}}\right),
\] 
where $\gamma_l^*\left(\tilde{\mathcal{N}}_B(\epsilon,\zeta)_{|\overline{\mathcal{U}_l}}\right)$ is the pull back of the restriction $\tilde{\mathcal{N}}_B(\epsilon,\zeta)_{|\overline{\mathcal{U}_l}}$ by the parametrization $\gamma_l$. Since the restriction map from $C^{m,\alpha}(\partial\Omega^i)$ to $C^{m,\alpha}(\overline{\mathcal{U}_l})$ and the pullback map $\gamma_l^*$ from $C^{m,\alpha}(\overline{\mathcal{U}_l})$ to $C^{m,\alpha}(\overline{B_{n-1}(0,1)})$ are linear and continuous and since $\mathcal{N}_B$ is real analytic from $]-\epsilon_0,\epsilon_0[ \times \R$ to $C^{m,\alpha}( \partial\Omega^i)$, it follows that the map $\mathcal{M}_A$ is real analytic from $]-\epsilon_0,\epsilon_0[ \times \R$ to $C^{m,\alpha}( \overline{B_{n-1}(0,1)})$. Then  Lemma \ref{lemmaboundA(ep,dot,dot)} implies that 
\begin{equation}\label{lemmaboundB(ep,dot,dot).eq4}
\sup_{\epsilon \in \mathcal{E}} \|A(\epsilon,\cdot,\cdot)\|_{C^{m,\alpha}(\overline{B_{n-1}(0,1)} \times \overline{\mathcal{J}})} \leq C
\end{equation}
for some $C>0$. Now the validity of \eqref{lemmaboundB(ep,dot,dot).eq2} follows by \eqref{lemmaboundB(ep,dot,dot).eq3} and \eqref{lemmaboundB(ep,dot,dot).eq4}. The proof is complete.
\end{proof}

\subsubsection{The auxiliary maps $N$ and $S$}

In the proof of our main Theorem \ref{Thmunisol2} we will exploit two auxiliary maps, which we denote by $N$ and $S$ and are defined as follows. Let $\epsilon'$ be as in Theorem \ref{existenceThm} (v). We denote by $N=(N_1,N_2,N_3)$ the map from $]-\epsilon',\epsilon'[ \times \COo \times \COi_0 \times \R \times \COi$ to $\COo \times \COi \times C^{0,\alpha}(\partial\Omega^i)$ defined by
\begin{align}
\label{N1}
\nonumber
& N_1[\epsilon,\phi^o, \phi^i, \zeta,\psi^i](x)  \equiv \left(\frac{1}{2}I + W_{\partial\Omega^o}\right)[\phi^o](x)
\\
& \quad  - \epsilon^{n-1} \int_{\partial\Omega^i}{\nu_{\Omega^i}(y) \cdot \nabla S_n(x-\epsilon y) \phi^i(y) \,d\sigma_y}  +  \epsilon^{n-2}\zeta S_n(x) 
&& \forall x \in \partial\Omega^o,
\\
\label{N2}
\nonumber
&N_2[\epsilon,\phi^o, \phi^i, \zeta,\psi^i](t) \equiv \left(-\frac{1}{2}I + W_{\partial\Omega^i}\right)[\phi^i](t) +\zeta\,S_n(t)
\\
&\quad + w^+_{\Omega^o}[\phi^o](\epsilon t) - (\partial_\zeta F) (0,t,\zeta^i) \left(\frac{1}{2}I + W_{\partial\Omega^i}\right)[\psi^i](t)
&& \forall t \in \partial\Omega^i, 
\\
\label{N3}\nonumber
&N_3[\epsilon,\phi^o, \phi^i, \zeta,\psi^i](t)\equiv \nu_{\Omega^i}(t) \cdot \left( \epsilon \nabla w^+_{\Omega^o}[\phi^o](\epsilon t)\right.\\
&\left.
\qquad\qquad\qquad\quad\qquad + \nabla w^-_{\Omega^i}[\phi^i](t) +  \zeta\nabla S_n(t)  - \nabla w^+_{\Omega^i}[\psi^i](t) \right)  
&& \forall t \in \partial\Omega^i,
\end{align}
for all $(\epsilon,  \phi^o, \phi^i, \zeta, \psi^i) \in ]-\epsilon',\epsilon'[  \times \COo \times \COi_0 \times \R \times \COi$ and we denote by  $S=(S_1,S_2,S_3)$ the map from $ ]-\epsilon',\epsilon'[ \times \COi$ to $\COo \times \COi \times C^{0,\alpha}(\partial\Omega^i)$ defined by
\begin{align}
\label{S1}
&S_1[\epsilon,\psi^i] (x) \equiv 0 &&  \forall x \in \partial \Omega^o,
\\
\label{S2}
\nonumber
&S_2[\epsilon,\psi^i] (t)  \equiv - t \cdot \nabla \mathfrak{u}^o(0) - \epsilon \mathfrak{u}^o(\epsilon t) + (\partial_\epsilon F) (0,t,\zeta^i)
\\ 
& \quad+ \epsilon \tilde{F}\left(\epsilon,t,\zeta^i,\left(\frac{1}{2}I + W_{\partial\Omega^i}\right) [\psi^i](t)\right) &&  \forall t \in \partial\Omega^i, 
\\
\label{S3}\nonumber
&S_3[\epsilon,\psi^i] (t)\\
&\equiv - \nu_{\Omega^i}(t) \cdot \nabla \mathfrak{u}^o(\epsilon t) + G \left(\epsilon,t,\epsilon\left(\frac{1}{2}I + W_{\partial\Omega^i}\right)[\psi^i](t) + \zeta^i \right)  &&  \forall t \in \partial\Omega^i,
\end{align}
for all $(\epsilon,\psi^i) \in ]-\epsilon',\epsilon'[ \times \COi$. For the maps $N$ and $S$ we have the following result.

\begin{prop}\label{N^-1}
	Let assumptions  \eqref{zetaicond}, \eqref{F1}, and \eqref{realanalhp} hold true. Then there exists $\epsilon'' \in \, ]0,\epsilon'[$ such that the following statements hold:
	\begin{enumerate}
		\item[(i)] For all fixed $\epsilon\in]-\epsilon'',\epsilon''[$ the operator $N[\epsilon,\cdot,\cdot,\cdot,\cdot]$ is a linear homeomorphism from $\COo \times \COi_0 \times \R \times \COi$ to $\COo \times \COi \times C^{0,\alpha}(\partial\Omega^i)$;
		\item[(ii)] The map from  $]-\epsilon'',\epsilon''[$ to $\mathcal{L}(\COo \times \COi \times C^{0,\alpha}(\partial\Omega^i),\\ \COo \times \COi_0 \times \R \times \COi)$ which takes $\epsilon$ to  $N[\epsilon,\cdot,\cdot,\cdot,\cdot]^{(-1)}$ is real analytic;
		\item[(iii)] Equation \eqref{Me=0.e1} is equivalent to 
		\begin{equation}\label{N^-1S}
		(\phi^o, \phi^i, \zeta,\psi^i) = N[\epsilon,\cdot,\cdot,\cdot,\cdot]^{(-1)} [ S[\epsilon,\psi^i] ]
		\end{equation}
		for all $(\epsilon,  \phi^o, \phi^i, \zeta, \psi^i) \in ]-\epsilon'',\epsilon''[  \times \COo \times \COi_0 \times \R \times \COi$.
	\end{enumerate}

\end{prop}
\begin{proof}
	By the definition of $N$ (cf. \eqref{N1}--\eqref{N3}), by the mapping properties of the double layer potential  (cf. Theorem \ref{doublepot} (ii)-(iii)) and of integral operators with real analytic kernels and no singularity (see, e.g., Lanza de Cristoforis and Musolino \cite{LaMu13}),  by assumption \eqref{realanalhp} (which implies that $(\partial_\zeta F) (0,\cdot,\zeta^i)$ belongs to $\COi$), and by standard calculus in Banach spaces, one verifies that the map  from $]-\epsilon',\epsilon'[$ to 
	\[
	\mathcal{L}(\COo \times \COi_0 \times \R \times \COi\,,\;\COo \times \COi \times C^{0,\alpha}(\partial\Omega^i))
	\]
	which takes $\epsilon$ to $N[\epsilon,\cdot,\cdot,\cdot,\cdot]$ is real analytic. Then one observes that  
	\begin{equation*}
	N[0,\phi^o, \phi^i, \zeta,\psi^i] = \partial_{(\phi^o,\phi^i,\zeta,\psi^i)} M[0,\phi^o_0, \phi^i_0, \zeta_0, \psi^i_0].(\phi^o, \phi^i, \zeta,\psi^i)
	\end{equation*}
	and thus Theorem \ref{existenceThm} (iv) implies that $N[0,\cdot,\cdot,\cdot,\cdot]$ is an isomorphism from $\COo \times \COi_0 \times \R \times \COi$ to $\COo \times \COi \times \COi$.
	Since the set of invertible operators is open in $\mathcal{L}(\COo \times \COi_0 \times \R \times \COi, \COo \times \COi \times C^{0,\alpha}(\partial\Omega^i))$ and since the map which takes a linear invertible operator to its inverse is real analytic (cf.~Hille and Phillips \cite{HiPi57}), we deduce the validity of $(i)$ and $(ii)$.   To prove $(iii)$ we observe that, by the definition of $N$ and $S$ in \eqref{N1}--\eqref{S3} it readily follows that \eqref{Me=0.e1} is equivalent to  
	\begin{equation*}
	N[\epsilon,\phi^o, \phi^i, \zeta,\psi^i] = S[\epsilon,\psi^i]\,.
	\end{equation*}
Then the validity of $(iii)$ is a consequence of $(i)$.
\end{proof}

\subsubsection{The main theorem}

We are now ready to prove our main Theorem \ref{Thmunisol2} on the local uniqueness of the solution   $(u^o_\epsilon,u^i_\epsilon)$ provided by Theorem \ref{uesol}.

\begin{teo}\label{Thmunisol2}
	Let assumptions  \eqref{zetaicond}, \eqref{F1}, and \eqref{realanalhp} hold true.  Let $\epsilon' \in ]0,\epsilon_0[$ be as in Theorem \ref{existenceThm} (v).	Let $\{(u^o_\epsilon,u^i_\epsilon)\}_{\epsilon \in ]0,\epsilon'[}$ be as in Theorem \ref{uesol}. Then there exist $\epsilon^\ast \in ]0,\epsilon'[$ and $\delta^\ast \in ]0,+\infty[$ such that the following property holds:
	
	If $\epsilon \in ]0,\epsilon^\ast[$ and $(v^o,v^i) \in C^{1,\alpha}(\Omega(\epsilon)) \times C^{1,\alpha}(\epsilon\Omega^i)$ is a solution of problem \eqref{princeq} with
	\[
	\left\| {v^i(\epsilon \cdot ) - u^i_\epsilon(\epsilon \cdot )} \right\| _{C^{1,\alpha}(\partial \Omega^i)} < \epsilon\delta^\ast,
    \]
	
	then 
	\[
	(v^o,v^i)=(u^o_\epsilon,u^i_\epsilon)\,.
	\]
\end{teo}

\begin{proof}
{\it $\bullet$ Step 1: Fixing $\epsilon^\ast$.}
\\
Let $\epsilon'' \in \, ]0,\epsilon'[$ be as in Proposition \ref{N^-1} and let $\epsilon'''\in]0,\epsilon''[$ be fixed.	By the compactness of $[-\epsilon''',\epsilon''']$ and by the continuity of the norm in $\mathcal{L}(\COi \times C^{0,\alpha}(\partial\Omega^i)  \times \COo, \COo \times \COi_0 \times \R \times \COi)$, there exists a real number $C_1>0$ such that 
\begin{equation}\label{unicconstantN-1}
\begin{split}
\| N[\epsilon,\cdot,\cdot,\cdot,\cdot]^{(-1)} \|_{\mathcal{L}(\COi \times C^{0,\alpha}(\partial\Omega^i)  \times \COo, \COo \times \COi_0 \times \R \times \COi)}&\\
 \le C_1&
\end{split}
\end{equation}
for all $\epsilon \in [-\epsilon''',\epsilon''']$ (see also Proposition \ref{N^-1} (ii)).
Let $U_0$ be the open neighborhood of $(\phi^o_0, \phi^i_0, \zeta_0, \psi^i_0)$ in $\COo \times \COi_0 \times \R \times \COi$ introduced in Theorem \ref{existenceThm} (v). Then we take $K>0$ such that 
\[
\overline{\mathcal{B}_{0,K}}\subseteq U_0 
\]
(see \eqref{B0r} for the definition of $\mathcal{B}_{0,K}$).
Since $(\Phi^o[\cdot],\Phi^i[\cdot],Z[\cdot],\Psi^i[\cdot])$ is continuous (indeed real analytic) from $]-\epsilon',\epsilon'[$ to $U_0$, there exists $\epsilon^* \in ]0,\epsilon'''[$ such that 
\begin{equation}\label{PhiBK/2}
(\Phi^o[\eta] , \Phi^i[\eta],Z[\eta] , \Psi^i[\eta])\in \mathcal{B}_{0,K/2}\subset U_0\qquad\forall \eta \in ]0,\epsilon^*[\,.
\end{equation}
Moreover, we assume that 
\begin{equation*}
\epsilon^\ast < 1.
\end{equation*} 
We will prove that the theorem holds for such choice of $\epsilon^\ast$. We observe that the condition $\epsilon^\ast <1$ is not really needed in the proof but simplifies many computations.
\\
	
{\it $\bullet$ Step 2: Planning our strategy.}
\\
We suppose that there exists a pair of functions $(v^o,v^i) \in C^{1,\alpha}(\overline{\Omega(\epsilon)}) \times C^{1,\alpha}(\epsilon \overline{\Omega^i})$ that is a solution of problem \eqref{princeq} for a certain $\epsilon \in ]0,\epsilon^\ast[$ (fixed) and such that
\begin{equation}\label{unidisuguiesp}
\left\| \frac{v^i(\epsilon \cdot ) - u^i_\epsilon(\epsilon \cdot )}{\epsilon} \right\| _{C^{1,\alpha}(\partial \Omega^i)} \leq \delta^\ast\,,
\end{equation}
for some $\delta^\ast \in ]0,+\infty[$.
Then, by Proposition \ref{rappresentsol}, there exists a unique quadruple $(\phi^o,\phi^i,\zeta,\psi^i) \in \COo \times \COi_0 \times \R \times \COi$ such that
	\begin{equation}\label{unicitauouirappr}
	\begin{aligned}
	& v^o = U^o_{\epsilon}[\phi^o,\phi^i,\zeta,\psi^i] &&  \text{in } \overline{\Omega(\epsilon)},
	\\
	& v^i = U^i_{\epsilon}[\phi^o,\phi^i,\zeta,\psi^i] &&  \text{in } \epsilon \overline{\Omega^i}.
	\end{aligned}
	\end{equation}
	We shall show that for $\delta^\ast$ small enough we have
	\begin{equation}\label{uniequalpsiPsi}
	(\phi^o,\phi^i,\zeta,\psi^i) = (\Phi^o[\epsilon] , \Phi^i[\epsilon],	Z[\epsilon] , \Psi^i[\epsilon])\,.
	\end{equation}
	Indeed, if we have \eqref{uniequalpsiPsi}, then Theorem \ref{uesol} would imply that
	\begin{equation*}
	(v^o,v^i) = (u^o_\epsilon,v^i_\epsilon),
	\end{equation*}
	and our proof would be completed.
	Moreover, to prove \eqref{uniequalpsiPsi} it suffices to show that
	\begin{equation}\label{unidensitiesinB0K}
	(\phi^o,\phi^i,\zeta,\psi^i) \in \mathcal{B}_{0,K} \subset U_0\,.
	\end{equation}
In fact, in that case, both $(\epsilon,\phi^o,\phi^i,\zeta,\psi^i)$ and $(\epsilon,\Phi^o[\epsilon] , \Phi^i[\epsilon],	Z[\epsilon] , \Psi^i[\epsilon])$ would stay in the zero set of $M$ (cf. Theorem \ref{existenceThm} (ii)) and thus  \eqref{unidensitiesinB0K}  together with \eqref{PhiBK/2} and Theorem \ref{existenceThm} (v) would imply \eqref{uniequalpsiPsi}.

	So, our aim is  now to prove that \eqref{unidensitiesinB0K} holds true for a suitable choice of $\delta^*>0$. It will be also convenient to restrict our search to 
	\[
	0<\delta^*<1.
	\]
	As for the condition $\epsilon^*<1$, this condition on $\delta^*$ is not really needed, but simplifies our computations. Then to find $\delta^*$ and prove \eqref{unidensitiesinB0K} we will proceed as follows. First  we  obtain an estimate for $\psi^i$ and $\Psi^i[\epsilon]$ with a bound that does not depend on $\epsilon$ and $\delta^*$. Then we use such estimate to show that
	\begin{equation*}
	\| S[\epsilon,\psi^i] - S[\epsilon,\Psi^i[\epsilon]] \|_{\COo \times \COi \times C^{0,\alpha}(\partial\Omega^i)}
	\end{equation*} is smaller than a constant times $\delta^*$, with a constant that does not depend on $\epsilon$ and $\delta^*$. We will split the analysis for $S_1$, $S_2$, and $S_3$ and we find convenient to study $S_3$ before $S_2$. Indeed, the computations for $S_2$ and $S_3$ are very similar but those for $S_3$ are much shorter and can serve better to illustrate the techniques employed. We also observe that the analysis for $S_2$ requires the study of other auxiliary functions $T_1$, $T_2$, and $T_3$ that we will introduce. Finally, we will exploit the estimate for $\| S[\epsilon,\psi^i] - S[\epsilon,\Psi^i[\epsilon]] \|_{\COo \times \COi \times C^{0,\alpha}(\partial\Omega^i)}$ to determine $\delta^*$ and prove \eqref{unidensitiesinB0K}.
\\

{\it $\bullet$ Step 3: Estimate for $\psi$ and $\Psi[\epsilon]$.}
\\
By condition \eqref{unidisuguiesp}, by the second equality in \eqref{unicitauouirappr}, by Theorem \ref{uesol}, and by arguing as in \eqref{(u^i-u_e^i)} and \eqref{psi^i} in Theorem \ref{Thmunisol1}, we obtain
\begin{equation}\label{inequality(1/2I + W)psiPsi}
\left\| \left(\frac{1}{2}I + W_{\partial\Omega^i}\right)[\psi^i] - \left(\frac{1}{2}I + W_{\partial\Omega^i}\right)[\Psi^i[\epsilon]] \right\|_{\COi} \leq \, \delta^\ast
\end{equation}
and
\begin{equation}\label{inequalitypsiPsi}
\begin{split}
\|& \psi^i - \Psi^i[\epsilon] \|_{\COi} 
\leq \left\|\left(\frac{1}{2}I + W_{\partial\Omega^i}\right)^{(-1)} \right\|_{\mathcal{L}(\COi,\COi)} 
\\
&\times\left\| \left(\frac{1}{2}I + W_{\partial\Omega^i}\right)[\psi^i] - \left(\frac{1}{2}I + W_{\partial\Omega^i}\right)[\Psi^i[\epsilon]] \right\|_{\COi} \leq C_2 \delta^\ast,
\end{split}
\end{equation}
where
\begin{equation*}
C_2 \equiv \left\|\left(\frac{1}{2}I + W_{\partial\Omega^i}\right)^{(-1)} \right\|_{\mathcal{L}(\COi,\COi)}.
\end{equation*}
By \eqref{PhiBK/2} we have
\begin{equation}\label{psi-Psi}
\|\psi^i_0 - \Psi^i[\eta]\|_{\COi} \leq \frac{K}{2} \qquad \forall \eta \in ]0,\epsilon^*[.
\end{equation}
Then, by \eqref{inequalitypsiPsi} and \eqref{psi-Psi}, and by the triangular inequality, we see that
	\begin{equation*}
		\begin{aligned}
		\|\psi^i \|_{\COi}
		&
		\leq  \|\psi^i_0\|_{\COi} + \|\psi^i - \Psi^i[\epsilon] \|_{\COi} + \|\Psi^i[\epsilon] - \psi^i_0\|_{\COi}
		\\
		&
		\leq \|\psi^i_0\|_{\COi} + C_2\, \delta^\ast + \frac{K}{2},
		\\
		\|\Psi^i[\epsilon]\|_{\COi}& \leq \|\psi^i_0\|_{\COi}+ \frac{K}{2}.
		\end{aligned}
	\end{equation*}
	Then, by taking $R_1 \equiv \|\psi^i_0\|_{\COi} + C_2+\frac{K}{2}$ and $R_2 \equiv \|\psi^i_0\|_{\COi} + \frac{K}{2}$ (and recalling that $\delta^\ast \in ]0,1[$), one verifies that
	\begin{equation}\label{uniccostantR_1R_2}
	\|\psi^i \|_{\COi} \leq R_1\qquad\text{and}\qquad
	\|\Psi^i[\epsilon]\|_{\COi} \leq R_2\,.
	\end{equation}
	We note here that both $R_1$ and $R_2$ do not depend on $\epsilon$ and $\delta^\ast$ as long they belong to $]0,\epsilon^\ast[$ and $]0,1[$, respectively.
	
	\smallskip
	{\it $\bullet$ Step 4: Estimate for $S_1$.}
	\\
	We now pass to estimate the norm
	\begin{equation*}
	\| S[\epsilon,\psi^i] - S[\epsilon,\Psi^i[\epsilon]] \|_{\COo \times \COi \times C^{0,\alpha}(\partial\Omega^i)}.
	\end{equation*} 
	To do so we consider separately $S_1$, $S_2$, and $S_3$. Since $S_1=0$ (cf. definition \eqref{S1}), we readily obtain that
	\begin{equation}\label{unicS1inqual}
	\| S_1[\epsilon, \psi^i] - S_1[\epsilon,\Psi^i[\epsilon]] \|_{\COo} = 0.
	\end{equation}
	
	\smallskip
	{\it $\bullet$ Step 5: Estimate for $S_3$.}
	\\
	We consider $S_3$ before $S_2$ because its treatment is simpler and more illustrative of the techniques used. By \eqref{S3} and by the Mean Value Theorem in Banach space (see, e.g., Ambrosetti and Prodi \cite[Thm. 1.8]{AmPr95}), we compute that
	\begin{equation}\label{|S3|}
		\begin{split}
		&\| S_3[\epsilon, \psi^i] - S_3[\epsilon,\Psi^i[\epsilon]] \|_{C^{0,\alpha}(\partial\Omega^i)} 
		\\
		&\quad = \left\| G \left(\epsilon,\cdot,\epsilon\left(\frac{1}{2}I + W_{\partial\Omega^i}\right)[\psi^i] + \zeta^i \right)\right.\\
		&\left.\qquad - G \left(\epsilon,\cdot,\epsilon\left(\frac{1}{2}I + W_{\partial\Omega^i}\right)[\Psi^i[\epsilon]] + \zeta^i \right) \right\|_{C^{0,\alpha}(\partial\Omega^i)}
		\\
		&\quad = \left\| \mathcal{N}_G\left(\epsilon, \epsilon\left(\frac{1}{2}I + W_{\partial\Omega^i}\right)[\psi^i] + \zeta^i\right)\right.\\
		&\left.\qquad - \mathcal{N}_G\left(\epsilon, \epsilon\left(\frac{1}{2}I + W_{\partial\Omega^i}\right)[\Psi^i[\epsilon]] + \zeta^i\right)\right\|_{C^{0,\alpha}(\partial\Omega^i)}
		\\
		&\quad
		\leq
		\left\|d_v \mathcal{N}_G (\epsilon,\tilde{\psi}^i) \right\|_{\mathcal{L}(\COi,C^{0,\alpha}(\partial\Omega^i))} \\
		&\qquad \times\epsilon\left\|\left(\frac{1}{2}I + W_{\partial\Omega^i}\right)[\psi^i] -  \left(\frac{1}{2}I + W_{\partial\Omega^i}\right)[\Psi^i[\epsilon]]\right\|_{C^{0,\alpha}(\partial\Omega^i)},
		\end{split}
	\end{equation}
	where 
	\begin{equation*}
	\tilde{\psi}^i = \theta \left(\epsilon\left(\frac{1}{2}I + W_{\partial\Omega^i}\right)[\psi^i] + \zeta^i \right) + (1-\theta) \left(\epsilon \left(\frac{1}{2}I + W_{\partial\Omega^i}\right)[\Psi^i[\epsilon]] +\zeta^i \right),
	\end{equation*}
	for some $\theta \in ]0,1[$.
	Then, by the membership of $\epsilon$ and $\theta$ in $]0,1[$ we have
	\[
	\begin{split}
	\|\tilde{\psi}^i \|_{\COi} \leq &\left\|\left(\frac{1}{2}I + W_{\partial\Omega^i}\right)[\psi^i]+ \zeta^i\right\|_{\COi}\\
	&  + \left\|\left(\frac{1}{2}I + W_{\partial\Omega^i}\right)[\Psi^i[\epsilon]] +\zeta^i\right\|_{\COi}
	\end{split}
	\]
and, by setting 
\[
	C_3 \equiv \left\|\frac{1}{2}I + W_{\partial\Omega^i} \right\|_{\mathcal{L}(\COi,\COi)}\,,
	\]
	we obtain
	\begin{equation}\label{boundtildepsiiesp}
	\|\tilde{\psi}^i \|_{\COi} \leq
	C_3\|\psi^i\|_{\COi} + C_3\|\Psi^i[\epsilon]\|_{\COi} + 2 |\zeta^i| 
	\leq R,	
	\end{equation}
	with
	\begin{equation}\label{uniccostantR}
	R \equiv C_3(R_1 + R_2) + 2|\zeta^i|
	\end{equation}
	which does not depend on $\epsilon$.
	We wish now to estimate the operator norm 
	\[
	\left\|d_v \mathcal{N}_G (\eta,\tilde{\psi}^i) \right\|_{\mathcal{L}(\COi,C^{0,\alpha}(\partial\Omega^i))}
	\]
	uniformly for $\eta \in ]0,\epsilon^\ast[$.
	However, we cannot exploit a compactness argument on $[0,\epsilon^*]\times \overline{B_{C^{1,\alpha}(\partial\Omega^i)}(0,R)}$, because $\overline{B_{C^{1,\alpha}(\partial\Omega^i)}(0,R)}$ is not compact in the infinite dimensional space $C^{1,\alpha}(\partial\Omega^i)$. Then we argue as follows. We observe that, by assumption \eqref{realanalhp}, the partial derivative $\partial_\zeta G(\eta,t,\zeta)$ exists for all $(\eta,t,\zeta) \in ]-\epsilon_0,\epsilon_0[ \times \partial\Omega^i \times \R$ and, by Proposition \ref{differenzialeN_H} and Lemma \ref{C0C1alphaalgebra} (i) in the Appendix, we obtain that
	\begin{equation}\label{norm.1}
	\begin{split}
		\|d_v \mathcal{N}_G (\eta,\tilde{\psi}^i) \|_{\mathcal{L}(\COi,C^{0,\alpha}(\partial\Omega^i))}& \leq 
		\|\mathcal{N}_{\partial_\zeta G} (\eta,\tilde{\psi}^i) \|_{C^{0,\alpha}(\partial\Omega^i)}
		\\
		& \leq
		\|\partial_\zeta G (\eta, \cdot, \tilde{\psi}^i (\cdot))\|_{C^{0,\alpha}(\partial\Omega^i)}
	\end{split}
	\end{equation}
	for all $\eta \in ]0,\epsilon^\ast[$. By Proposition \ref{proponcompC1alpha} (ii) in the Appendix, there exists $C_4 >0$ such that
	\begin{equation}\label{norm.1.1}
	\begin{split}
	&\|\partial_\zeta G (\eta, \cdot, \tilde{\psi}^i (\cdot))\|_{C^{0,\alpha}(\partial\Omega^i)}\\
	 &\quad\le C_4 \|\partial_\zeta G (\eta, \cdot,\cdot)\|_{C^{0,\alpha}(\partial\Omega \times [-R,R])} \left(1 + \|\tilde{\psi}^i\|^\alpha_{\COi} \right) \quad \forall \eta \in ]0,\epsilon^\ast[.
	\end{split}
	\end{equation}	
	Moreover, by assumption \eqref{realanalhp} one deduces that the map $\tilde{\mathcal{N}}_G$ defined as in Lemma \ref{lemmaboundB(ep,dot,dot)} (with $B=G$) is real analytic from $]-\epsilon_0,\epsilon_0[ \times \R$ to $C^{0,\alpha}(\partial\Omega^i)$ and, by Proposition \ref{differenzialeN_H}, one has that $\partial_\zeta \tilde{\mathcal{N}}_G = \tilde{\mathcal{N}}_{\partial_\zeta G}$. Hence, by Lemma \ref{lemmaboundB(ep,dot,dot)} (with $m=0$), there exists $C_5>0$ (which  does not depend on $\epsilon \in ]0,\epsilon^\ast[$ and $\delta^\ast \in ]0,1[$) such that
	\begin{equation}\label{norm.2}
		\sup_{\eta \in [-\epsilon^\ast,\epsilon^\ast]} \|\partial_\zeta G (\eta, \cdot,\cdot)\|_{C^{0,\alpha}(\partial\Omega^i \times [-R,R])} \leq C_5.
	\end{equation}
    Hence, by \eqref{boundtildepsiiesp}, \eqref{norm.1}, \eqref{norm.1.1} and \eqref{norm.2}, we deduce that
	\begin{equation}\label{norm.3}
	\|d_v \mathcal{N}_G (\epsilon,\tilde{\psi}^i) \|_{\mathcal{L}(\COi,C^{0,\alpha}(\partial\Omega^i))} \le C_4\, C_5\,(1+ R^\alpha).
	\end{equation}
    By \eqref{inequality(1/2I + W)psiPsi}, \eqref{|S3|}, and \eqref{norm.3}, and by the membership of $\epsilon$ in $]0,\epsilon^\ast[ \subset \,]0,1[$, we obtain that
	\begin{equation}\label{unicS2inequal}
	\| S_3[\epsilon, \psi^i] - S_3[\epsilon,\Psi^i[\epsilon]] \|_{C^{0,\alpha}(\partial\Omega^i)} 
	\leq C_4\, C_5\, (1+ R^\alpha)\; \delta^\ast.
	\end{equation}
		
	\smallskip
	{\it $\bullet$ Step 6: Estimate for $S_2$.}
	\\
	Finally, we consider $S_2$. By \eqref{S2} and by the fact that $\epsilon \in]0,1[$, we have
	{\small\begin{equation}\label{s1-S1<A123}
		\begin{split}
		&\| S_2[\epsilon, \psi^i] - S_2[\epsilon,\Psi^i[\epsilon]] \|_{\COi} 
		\\
		& \leq \epsilon \left\|\tilde{F}\left(\epsilon, \cdot,\zeta^i,\left(\frac{1}{2}I + W_{\partial\Omega^i}\right)[\psi^i] \right) - \tilde{F}\left(\epsilon,\cdot,\zeta^i,\left(\frac{1}{2}I + W_{\partial\Omega^i}\right)[\Psi^i[\epsilon]] \right) \right\|_{\COi}
		\\
		& \leq \left\| \int_{0}^{1} (1-\tau) \left\{  T_1[\epsilon, \psi^i, \Psi^i[\epsilon]](\tau,\cdot) +  2 \, T_2[\epsilon, \psi^i, \Psi^i[\epsilon]](\tau,\cdot) \right.\right.\\
		&\left.\left.\qquad\qquad\qquad\qquad\qquad\qquad\qquad\qquad\qquad
		+ T_3[\epsilon, \psi^i, \Psi^i[\epsilon]](\tau,\cdot) \right\}  \,d\tau \right\|_{\COi},
		\end{split}
	\end{equation}}
	where $T_1[\epsilon, \psi^i, \Psi^i[\epsilon]]$, $T_2[\epsilon, \psi^i, \Psi^i[\epsilon]]$, and $T_3[\epsilon, \psi^i, \Psi^i[\epsilon]]$ are the functions from $]0,1[ \times \partial\Omega^i$ to $\R$ defined by
	\begin{align}
	\label{A1}
	    \nonumber
		&T_1[\epsilon, \psi^i, \Psi^i[\epsilon]](\tau,t)\\
		\nonumber
		 &\equiv  (\partial^2_\epsilon F)\left(\tau\epsilon,t, \tau\epsilon \left(\frac{1}{2}I + W_{\partial\Omega^i}\right)[\psi^i](t) + \zeta^i \right) 
		\\ &
		\quad- (\partial^2_\epsilon F)\left(\tau\epsilon,t, \tau\epsilon \left(\frac{1}{2}I + W_{\partial\Omega^i}\right)[\Psi^i[\epsilon]](t) + \zeta^i \right) ,
		\\
		\label{A2}
		\nonumber
		&T_2[\epsilon, \psi^i, \Psi^i[\epsilon]](\tau,t)\\
		\nonumber
		&\equiv \left(\frac{1}{2}I + W_{\partial\Omega^i}\right)[\psi^i](t) \, (\partial_\epsilon \partial_\zeta F)\left(\tau\epsilon,t, \tau\epsilon \left(\frac{1}{2}I + W_{\partial\Omega^i}\right)[\psi^i](t) + \zeta^i \right) 
		\\&
		\quad- \left(\frac{1}{2}I + W_{\partial\Omega^i}\right)[\Psi^i[\epsilon]](t)
		(\partial_\epsilon \partial_\zeta F)\left(\tau\epsilon,t,\tau\epsilon \left(\frac{1}{2}I + W_{\partial\Omega^i}\right)[\Psi^i[\epsilon]](t) + \zeta^i \right), 
		\\
		\label{A3}
	    \nonumber
		&T_3[\epsilon, \psi^i, \Psi^i[\epsilon]](\tau,t)\\
		\nonumber
		&\equiv \left(\frac{1}{2}I + W_{\partial\Omega^i}\right)[\psi^i]^2(t) \, (\partial^2_\zeta F)\left(\tau\epsilon,t, \tau\epsilon \left(\frac{1}{2}I + W_{\partial\Omega^i}\right)[\psi^i] (t) +\zeta^i \right)
		\\
		&
		\quad- \left(\frac{1}{2}I + W_{\partial\Omega^i}\right)[\Psi^i[\epsilon]]^2(t)
		(\partial^2_\zeta F)\left(\tau\epsilon,t,\tau\epsilon \left(\frac{1}{2}I + W_{\partial\Omega^i}\right) [\Psi^i[\epsilon] ] (t) + \zeta^i \right)
		\end{align}
    for every $(\tau,t) \in ]0,1[ \times \partial \Omega^i$.
	
	We now want to bound the $C^{1,\alpha}$ norm with respect to the variable $t \in \partial\Omega^i$ of \eqref{A1}, \eqref{A2}, and \eqref{A3} uniformly with respect to $\tau \in ]0,1[$. 
	
	\smallskip
	{\it $\bullet$ Step 6.1: Estimate for $T_1$.}
	\\
	First we consider $T_1[\epsilon, \psi^i, \Psi^i[\epsilon]](\tau,\cdot)$. By the Mean Value Theorem in Banach space (see, e.g., Ambrosetti and Prodi \cite[Thm. 1.8]{AmPr95}),
	we can estimate the $\COi$ norm of $T_1[\epsilon, \psi^i, \Psi^i[\epsilon]](\tau,\cdot)$ (cf.~\eqref{A1}) as follows:
	\begin{equation}\label{unicA1equationnorm}
	\begin{split}
	&\|T_1[\epsilon, \psi^i, \Psi^i[\epsilon]](\tau,\cdot)\|_{\COi}
	\\
	&
	\qquad= \left\| \mathcal{N}_{\partial^2_\epsilon F}\left(\tau\epsilon, \tau\epsilon\left(\frac{1}{2}I + W_{\partial\Omega^i}\right)[\psi^i] + \zeta^i\right)\right.\\
	&\left.\qquad\qquad - \mathcal{N}_{\partial^2_\epsilon F}\left(\tau\epsilon, \tau\epsilon\left(\frac{1}{2}I + W_{\partial\Omega^i}\right)[\Psi^i[\epsilon]] + \zeta^i\right)\right\|_{\COi}
	\\
	&
	\qquad\leq
	\left\|d_v \mathcal{N}_{\partial^2_\epsilon F} (\tau\epsilon,\tilde{\psi}^i_1)\right\|_{\mathcal{L}(\COi,\COi)} \\
	&\qquad\qquad\times\tau\epsilon \left\|\left(\frac{1}{2}I + W_{\partial\Omega^i}\right)[\psi^i] -  \left(\frac{1}{2}I + W_{\partial\Omega^i}\right)[\Psi^i[\epsilon]]\right\|_{\COi}
	\end{split}
	\end{equation}
	where 
	\begin{equation*}
	\begin{split}
	&\tilde{\psi}^i_{1}\\
	& = \theta_1 \left(\tau\epsilon\left(\frac{1}{2}I + W_{\partial\Omega^i}\right)[\psi^i] + \zeta^i\right) + (1-\theta_1) \left(\tau\epsilon \left(\frac{1}{2}I + W_{\partial\Omega^i}\right)[\Psi^i[\epsilon]] +\zeta^i\right),
	\end{split}
	\end{equation*}
	for some $\theta_1 \in ]0,1[$. 
	
	\smallskip
	{\it $\bullet$ Step 6.2: Estimate for $T_2$.}
	\\
	We now consider $T_2[\epsilon, \psi^i, \Psi^i[\epsilon]](\tau,\cdot)$. 
	Adding and subtracting 
	\begin{equation*}
	\left(\frac{1}{2}I + W_{\partial\Omega^i}\right)[\Psi^i[\epsilon]] \, (\partial_\epsilon \partial_\zeta F)\left(\tau\epsilon,t,\tau\epsilon \left(\frac{1}{2}I + W_{\partial\Omega^i}\right)[\psi^i](t) + \zeta^i \right)
	\end{equation*} 
	in the right hand side of \eqref{A2} and using the triangular inequality, we obtain
	{\small \begin{equation}\label{unicA2equationnorm1}
	\begin{split}
	\|&T_2[\epsilon, \psi^i, \Psi^i[\epsilon]](\tau,\cdot)\|_{\COi}\\
	&
	\leq \left\| \left(\frac{1}{2}I + W_{\partial\Omega^i}\right)[\psi^i] \, (\partial_\epsilon \partial_\zeta F)\left(\tau\epsilon,\cdot,\tau\epsilon \left(\frac{1}{2}I + W_{\partial\Omega^i}\right)[\psi^i] + \zeta^i\right) 
	\right.
	\\ 
	&
	\left.
	- \left(\frac{1}{2}I + W_{\partial\Omega^i}\right)[\Psi^i[\epsilon]] \, (\partial_\epsilon \partial_\zeta F)\left(\tau\epsilon,\cdot,\tau\epsilon \left(\frac{1}{2}I + W_{\partial\Omega^i}\right)[\psi^i] +\zeta^i \right) \right\|_{\COi}
	\\
	&
	 + 
	\left\|
	\left(\frac{1}{2}I + W_{\partial\Omega^i}\right)[\Psi^i[\epsilon]] \, (\partial_\epsilon \partial_\zeta F)\left(\tau\epsilon,\cdot,\tau\epsilon \left(\frac{1}{2}I + W_{\partial\Omega^i}\right)[\psi^i] + \zeta^i \right)
	\right.
	\\ 
	&
	\left.
	-\left(\frac{1}{2}I + W_{\partial\Omega^i}\right)[\Psi^i[\epsilon]]
	(\partial_\epsilon \partial_\zeta F)\left(\tau\epsilon,\cdot,\tau\epsilon \left(\frac{1}{2}I + W_{\partial\Omega^i}\right)[\Psi^i[\epsilon]] +\zeta^i  \right) 
	\right\|_{\COi}.
	\end{split}
	\end{equation}}
	By Lemma \ref{C0C1alphaalgebra} in the Appendix and by the Mean Value Theorem in Banach space (see, e.g., Ambrosetti and Prodi \cite[Thm. 1.8]{AmPr95}),
	we can estimate the $\COi$ norm of $T_2[\epsilon, \psi^i, \Psi^i[\epsilon]](\tau,\cdot)$ (cf.~\eqref{A2} and \eqref{unicA2equationnorm1}) as follows:
	\begin{equation}\label{unicA2equationnorm}
	\begin{split}
	&\|T_2[\epsilon, \psi^i, \Psi^i[\epsilon]](\tau,\cdot)\|_{\COi}
	\\
	&	
	\leq
	2 \left\| (\partial_\epsilon \partial_\zeta F)\left(\tau\epsilon,\cdot,\tau\epsilon \left(\frac{1}{2}I + W_{\partial\Omega^i}\right)[\psi^i] + \zeta^i\right)  \right\|_{\COi} 
	\\
	&
	\quad\times
	\left\|\left(\frac{1}{2}I + W_{\partial\Omega^i}\right)[\psi^i] -  \left(\frac{1}{2}I + W_{\partial\Omega^i}\right)[\Psi^i[\epsilon]]\right\|_{\COi}
	\\
	&	
	\quad+ 2 \left\|\left(\frac{1}{2}I + W_{\partial\Omega^i}\right)[\Psi^i[\epsilon]]\right\|_{\COi}
	\\
	&
	\quad\times\left\| \mathcal{N}_{\partial_\epsilon\partial_\zeta F}\left(\tau\epsilon, \tau\epsilon\left(\frac{1}{2}I + W_{\partial\Omega^i}\right)[\psi^i] + \zeta^i\right)\right.\\
	&\left.\qquad\qquad - \mathcal{N}_{\partial_\epsilon\partial_\zeta F}\left(\tau\epsilon, \tau\epsilon\left(\frac{1}{2}I + W_{\partial\Omega^i}\right)[\Psi^i[\epsilon]] + \zeta^i \right)\right\|_{\COi}
	\\
	&
	\leq
	2 \left\| (\partial_\epsilon \partial_\zeta F)\left(\tau\epsilon,\cdot,\tau\epsilon \left(\frac{1}{2}I + W_{\partial\Omega^i}\right)[\psi^i] + \zeta^i \right) \right\|_{\COi} 
	\\&
	\quad\times
	\left\|\left(\frac{1}{2}I + W_{\partial\Omega^i}\right)[\psi^i] -  \left(\frac{1}{2}I + W_{\partial\Omega^i}\right)[\Psi^i[\epsilon]]\right\|_{\COi}
	\\
	&	 
	\quad+ 2\, C_3\, \| \Psi^i[\epsilon] \|_{\COi}
	\left\|d_v \mathcal{N}_{\partial_\epsilon\partial_\zeta F} (\tau\epsilon,\tilde{\psi}^i_{2}) \right\|_{\mathcal{L}(\COi,\COi)} 
	\\
	&
    \quad\times
	\tau\epsilon
	\left\|\left(\frac{1}{2}I + W_{\partial\Omega^i}\right)[\psi^i] -  \left(\frac{1}{2}I + W_{\partial\Omega^i}\right)[\Psi^i[\epsilon]]\right\|_{\COi},
	\end{split}
	\end{equation}
	where 
	\begin{equation*}
	\begin{split}
	\tilde{\psi}^i_{2}=& \theta_2 \left(\tau\epsilon\left(\frac{1}{2}I + W_{\partial\Omega^i}\right)[\psi^i] + \zeta^i\right)\\ 
	& + (1-\theta_2) \left(\tau\epsilon \left(\frac{1}{2}I + W_{\partial\Omega^i}\right)[\Psi^i[\epsilon]] +\zeta^i\right),
	\end{split}
	\end{equation*}
	for some $\theta_2 \in ]0,1[$. 
	
	\smallskip
	{\it $\bullet$ Step 6.3: Estimate for $T_3$.}
	\\
	Finally we consider $T_3[\epsilon, \psi^i, \Psi^i[\epsilon]](\tau,\cdot)$. Adding and subtracting the term
	\begin{equation*}
	\left(\frac{1}{2}I + W_{\partial\Omega^i}\right)[\Psi^i[\epsilon]]^2 \, (\partial^2_\zeta F)\left(\tau\epsilon,t, \tau\epsilon \left(\frac{1}{2}I + W_{\partial\Omega^i}\right)[\psi^i](t) + \zeta^i \right)
	\end{equation*}
	in the right hand side of \eqref{A3} and using the triangular inequality, we obtain
	\begin{equation}\label{unicA3equationnorm1}
	\begin{split}
	\|&T_3[\epsilon, \psi^i, \Psi^i[\epsilon]](\tau,\cdot)\|_{\COi}\\
	&
	\leq
	\left\|
	\left(\frac{1}{2}I + W_{\partial\Omega^i}\right)[\psi^i]^2 \, (\partial^2_\zeta F)\left(\tau\epsilon,\cdot,\tau\epsilon \left(\frac{1}{2}I + W_{\partial\Omega^i}\right)[\psi^i ] + \zeta^i\right)
	\right.
	\\
	&
	\left.
	- \left(\frac{1}{2}I + W_{\partial\Omega^i}\right)[\Psi^i[\epsilon]]^2 \, (\partial^2_\zeta F)\left(\tau\epsilon,\cdot,\tau\epsilon \left(\frac{1}{2}I + W_{\partial\Omega^i}\right)[\psi^i] + \zeta^i\right)
	\right\|_{\COi}
	\\
	&
	+ \left\|
	\left(\frac{1}{2}I + W_{\partial\Omega^i}\right)[\Psi^i[\epsilon]]^2 \, (\partial^2_\zeta F)\left(\tau\epsilon,\cdot,\tau\epsilon \left(\frac{1}{2}I + W_{\partial\Omega^i}\right)[\psi^i ]+\zeta^i \right)
	\right.
	\\
	&
	\left.- \left(\frac{1}{2}I + W_{\partial\Omega^i}\right)[\Psi^i[\epsilon]]^2
	(\partial^2_\zeta F)\left(\tau\epsilon,\cdot,\tau\epsilon \left(\frac{1}{2}I + W_{\partial\Omega^i}\right) [\Psi^i[\epsilon] ] +\zeta^i\right)\right\|_{\COi}\,.
	\end{split}
	\end{equation}
	By Lemma \ref{C0C1alphaalgebra} in the Appendix and by the Mean Value Theorem in Banach space (see, e.g., Ambrosetti and Prodi \cite[Thm. 1.8]{AmPr95}),
	we can estimate the $\COi$ norm of $T_3[\epsilon, \psi^i, \Psi^i[\epsilon]](\tau,\cdot)$ (cf. \eqref{A3} and \eqref{unicA3equationnorm1}) as follows:
	\begin{equation}\label{unicA3equationnorm}
	\begin{split}
	\|T_3&[\epsilon, \psi^i, \Psi^i[\epsilon]](\tau,\cdot)\|_{\COi} 
	\\
	&
	\leq
	2 \left\| (\partial^2_\zeta F)\left(\tau\epsilon,\cdot, \tau\epsilon \left(\frac{1}{2}I + W_{\partial\Omega^i}\right)[\psi^i] + \zeta^i\right) \right\|_{\COi} 
	\\
	&
	\quad\times\left\|\left(\frac{1}{2}I + W_{\partial\Omega^i}\right)[\psi^i]^2 - \left(\frac{1}{2}I + W_{\partial\Omega^i}\right)[\Psi^i[\epsilon]]^2 \right\|_{\COi}
	\\
	&
	\quad+ 2 \left\|\left(\frac{1}{2}I + W_{\partial\Omega^i}\right)[\Psi^i[\epsilon]]^2\right\|_{\COi} 
	\\
	&
	\quad\times\left\| \mathcal{N}_{\partial^2_\zeta F}\left(\tau\epsilon, \tau\epsilon\left(\frac{1}{2}I + W_{\partial\Omega^i}\right)[\psi^i] +\zeta^i \right)\right.\\
	&
	\left.\qquad - \mathcal{N}_{\partial^2_\zeta F}\left(\tau\epsilon, \tau\epsilon\left(\frac{1}{2}I + W_{\partial\Omega^i}\right)[\Psi^i[\epsilon]] + \zeta^i \right)\right\|_{\COi}
	\\
	&
	\leq
	4 \left\| (\partial^2_\zeta F)\left(\tau\epsilon,\cdot,\tau\epsilon \left(\frac{1}{2}I + W_{\partial\Omega^i}\right)[\psi^i] + \zeta^i \right) \right\|_{\COi} 
	\\
	&
	\quad\times\left\|\left(\frac{1}{2}I + W_{\partial\Omega^i}\right)[\psi^i] - \left(\frac{1}{2}I + W_{\partial\Omega^i}\right)[\Psi^i[\epsilon]] \right\|_{\COi}
	\\
	&
	\quad\times\left\|\left(\frac{1}{2}I + W_{\partial\Omega^i}\right)[\psi^i] + \left(\frac{1}{2}I + W_{\partial\Omega^i}\right)[\Psi^i[\epsilon]] \right\|_{\COi}
	\\
	&
	\quad+ 4 C_3^2 \, \|\Psi^i[\epsilon]\|^2_{\COi}
	\left\|d_v \mathcal{N}_{\partial^2_\zeta F} (\tau\epsilon,\tilde{\psi}^i_{3}) \right\|_{\mathcal{L}(\COi,\COi)} 
	\\
	&
	\quad\times
	\tau \epsilon \left\|\left(\frac{1}{2}I + W_{\partial\Omega^i}\right)[\psi^i] - \left(\frac{1}{2}I + W_{\partial\Omega^i}\right)[\Psi^i[\epsilon]] \right\|_{\COi},
	\end{split}
	\end{equation}	
	where 
	\begin{equation*}
	\begin{split}
	\tilde{\psi}^i_{3} =& \theta_3 \left(\tau\epsilon\left(\frac{1}{2}I + W_{\partial\Omega^i}\right)[\psi^i] + \zeta^i\right)\\
	& + (1-\theta_3) \left(\tau\epsilon \left(\frac{1}{2}I + W_{\partial\Omega^i}\right)[\Psi^i[\epsilon]] +\zeta^i\right),
	\end{split}
	\end{equation*}
	for some $\theta_3 \in ]0,1[$. 
	Let $R$ be as in  \eqref{uniccostantR}. By the same argument used to prove \eqref{boundtildepsiiesp}, one verifies the inequalities
	\begin{equation}\label{boundoverpsiiesp}
	\|\tilde{\psi}^i_{1}\|_{\COi} \leq R, \quad 
	\|\tilde{\psi}^i_{2}\|_{\COi} \leq R, \quad
	\|\tilde{\psi}^i_{3}\|_{\COi} \leq R.
	\end{equation}
	
	By assumption \eqref{realanalhp}, the partial derivatives
	$\partial_\zeta\partial^2_\epsilon F (\eta,t,\zeta)$, $\partial_\zeta\partial_\epsilon\partial_\zeta F (\eta,t,\zeta)$ and $\partial_\zeta\partial^2_\zeta F (\eta,t,\zeta)$ exist for all $(\eta,t,\zeta) \in ]-\epsilon_0,\epsilon_0[ \times \partial\Omega^i \times \R$ and by Proposition \ref{differenzialeN_H} and Lemma \ref{C0C1alphaalgebra} (ii) in the Appendix, we obtain
	\begin{equation}\label{norms.1}
	\begin{aligned}	
	\| d_v \mathcal{N}_{\partial^2_\epsilon F } (\tau\eta, \tilde{\psi}^i_{1}) \|_{\mathcal{L}(\COi,\COi)} &\leq 2 \| \mathcal{N}_{\partial_\zeta\partial^2_\epsilon F } (\tau\eta, \tilde{\psi}^i_{1}) \|_{\COi}
	\\ 
	&
	=
	2\, \|\partial_\zeta \partial^2_\epsilon F(\tau\eta, \cdot,  \tilde{\psi}^i_{1}(\cdot) ) \|_{C^{1,\alpha}(\partial\Omega^i)}\,,
	\\
	\| d_v \mathcal{N}_{\partial_\epsilon \partial_\zeta F } (\tau\eta, \tilde{\psi}^i_{2}) \|_{\mathcal{L}(\COi,\COi)} & \leq 2\| \mathcal{N}_{\partial_\zeta\partial_\epsilon \partial_\zeta F } (\tau\eta, \tilde{\psi}^i_{2}) \|_{\COi}
	\\
	& =  2\, \|\partial_\zeta \partial_\epsilon \partial_\zeta F(\tau\eta, \cdot, \tilde{\psi}^i_{2}(\cdot)) \|_{C^{1,\alpha}(\partial\Omega^i)}\,,
	\\
	\| d_v \mathcal{N}_{\partial^2_\zeta F } (\tau\eta, \tilde{\psi}^i_{3}) \|_{\mathcal{L}(\COi,\COi)} &\leq  \| \mathcal{N}_{\partial_\zeta\partial^2_\zeta F } (\tau\eta, \tilde{\psi}^i_{3}) \|_{\COi}
	\\
	& = 2\|\partial_\zeta \partial^2_\zeta F(\tau\eta, \cdot, \tilde{\psi}^i_{3}(\cdot)) \|_{\COi}\,,
	\end{aligned}
	\end{equation}
	for all $ \eta \in ]0,\epsilon^\ast[$.
	By Proposition \ref{proponcompC1alpha} (ii) in the Appendix, there exists $C_6 > 0$ such that 
	\begin{equation}\label{norms.1.1}
	\begin{aligned}
	&\|\partial_\zeta \partial^2_\epsilon F(\tau\eta, \cdot,  \tilde{\psi}^i_{1}(\cdot) ) \|_{C^{1,\alpha}(\partial\Omega^i)}\\
	 &\qquad\qquad\leq  C_6 \|\partial_\zeta \partial^2_\epsilon F(\tau\eta, \cdot, \cdot) \|_{C^{1,\alpha}(\partial\Omega \times [-R,R])} \left(1+ \|\tilde{\psi}^i_{1}\|_{\COi}\right)^2,
	\\
	&\|\partial_\zeta \partial_\epsilon \partial_\zeta F(\tau\eta, \cdot,  \tilde{\psi}^i_{2}(\cdot) ) \|_{\COi}\\
	 &\qquad\qquad\leq C_6 \|\partial_\zeta \partial_\epsilon \partial_\zeta F(\tau\eta, \cdot, \cdot ) \|_{C^{1,\alpha}(\partial\Omega^i \times [-R,R])} \left(1+ \|\tilde{\psi}^i_{2}\|_{\COi}\right)^2,
	\\
	&\|\partial_\zeta \partial^2_\zeta F(\tau\eta, \cdot,  \tilde{\psi}^i_{3}(\cdot) ) \|_{\COi}\\
	 &\qquad\qquad\leq C_6 \|\partial_\zeta \partial^2_\zeta F(\tau\eta, \cdot, \cdot ) \|_{C^{1,\alpha}(\partial\Omega^i \times [-R,R])} \left(1+ \|\tilde{\psi}^i_{3}\|_{\COi}\right)^2,
	\end{aligned}
	\end{equation}
	for all $ \eta \in ]0,\epsilon^\ast[$.	Now, by assumption \eqref{realanalhp} one deduces  that the map $\tilde{\mathcal{N}}_F$ defined as in Lemma \ref{lemmaboundB(ep,dot,dot)} (with $B=F$)   is real analytic from $]-\epsilon_0,\epsilon_0[ \times \R$ to $\COi$.  Then one verifies that also the maps
	\begin{equation*}
	\begin{split}
	&\partial^2_\epsilon \partial_\zeta \tilde{\mathcal{N}}_F = \tilde{\mathcal{N}}_{\partial^2_\epsilon \partial_\zeta F}, \quad \partial_\epsilon \partial^2_\zeta \tilde{\mathcal{N}}_F = \tilde{\mathcal{N}}_{\partial_\epsilon \partial^2_\zeta F},\quad \partial^3_\zeta \tilde{\mathcal{N}}_F = \tilde{\mathcal{N}}_{\partial^3 F}\,,\\
	&\partial_\epsilon \partial_\zeta \tilde{\mathcal{N}}_F = \tilde{\mathcal{N}}_{\partial_\epsilon \partial_\zeta F}, \quad \partial^2_\zeta \tilde{\mathcal{N}}_F = \tilde{\mathcal{N}}_{\partial^2_\zeta F}
	\end{split}
	\end{equation*}
	are real analytic from $]-\epsilon_0,\epsilon_0[ \times \R$ to $\COi$. Hence, Lemma \ref{lemmaboundB(ep,dot,dot)} (with $m=1$) implies that there exists $C_7>0$ such that
	\begin{equation}\label{norms.2}
	\begin{aligned}
	&\sup_{\eta \in [-\epsilon^\ast,\epsilon^\ast]} 
	\|\partial_\zeta \partial^2_\epsilon F(\tau\eta, \cdot, \cdot) \|_{C^{1,\alpha}(\partial\Omega \times [-R,R])} \leq C_7,
	\\
	&\sup_{\eta \in [-\epsilon^\ast,\epsilon^\ast]} 
	\|\partial_\zeta \partial_\epsilon \partial_\zeta F(\tau\eta, \cdot, \cdot) \|_{C^{1,\alpha}(\partial\Omega \times [-R,R])} \leq C_7,
	\\
	&\sup_{\eta \in [-\epsilon^\ast,\epsilon^\ast]} 
	\|\partial_\zeta \partial^2_\zeta F(\tau\eta, \cdot, \cdot) \|_{C^{1,\alpha}(\partial\Omega \times [-R,R])} \leq C_7,
	\\
	&\sup_{\eta \in [-\epsilon^\ast,\epsilon^\ast]} 
	\|\partial_\epsilon \partial_\zeta F(\tau\eta, \cdot, \cdot) \|_{C^{1,\alpha}(\partial\Omega \times [\zeta^i - C_3R,\zeta^i + C_3R])} \leq C_7,
	\\
	&\sup_{\eta \in [-\epsilon^\ast,\epsilon^\ast]} 
	\|\partial^2_\zeta F(\tau\eta, \cdot, \cdot) \|_{C^{1,\alpha}(\partial\Omega \times [\zeta^i - C_3R,\zeta^i + C_3R])} \leq C_7.
	\end{aligned}
	\end{equation}
	Thus, by \eqref{boundoverpsiiesp}, \eqref{norms.1}, \eqref{norms.1.1}, and \eqref{norms.2}, and by the membership of $\epsilon \in ]0,\epsilon^\ast[$ and $\delta^\ast \in ]0,1[$, we have
	\begin{equation}\label{unicboundNe^2FezFz^2F}
	\begin{aligned}
	\| d_v \mathcal{N}_{\partial^2_\epsilon F } (\tau\epsilon, \tilde{\psi}^i_{1}) \|_{\mathcal{L}(\COi,\COi)} &\leq 2 C_6\, C_7\,(1+R)^2\,,
	\\
	\| d_v \mathcal{N}_{\partial_\epsilon \partial_\zeta F } (\tau\epsilon, \tilde{\psi}^i_{2}) \|_{\mathcal{L}(\COi,\COi)}& \leq  2 C_6\, C_7\,(1+R)^2\,,
	\\
	\| d_v \mathcal{N}_{\partial^2_\zeta F } (\tau\epsilon, \tilde{\psi}^i_{3}) \|_{\mathcal{L}(\COi,\COi)}& \leq  2 C_6\, C_7\,(1+R)^2\,,
	\end{aligned}
	\end{equation}
	uniformly with respect to $\tau \in ]0,1[$.
	
	We can now bound the $C^{1,\alpha}$ norms with respect to the variable $t \in \partial\Omega^i$ of \eqref{A1}, \eqref{A2}, and \eqref{A3} uniformly with respect to $\tau \in ]0,1[$. Indeed, by \eqref{inequality(1/2I + W)psiPsi}, \eqref{unicA1equationnorm} and \eqref{unicboundNe^2FezFz^2F}, we obtain
	\begin{equation}\label{unicboundA1}
	\|T_1[\epsilon, \psi^i, \Psi^i[\epsilon]](\tau,\cdot)\|_{\COi} \leq  2 C_6\, C_7\,(1+R)^2\delta^\ast
	\end{equation}
	for all $\tau\in]0,1[$.
	By Proposition \ref{proponcompC1alpha} (ii) in the Appendix, by \eqref{uniccostantR_1R_2} and \eqref{norms.2}, we obtain 
	\begin{equation}\label{unicboundezFz^2F}
	\begin{aligned}
	&\left\|(\partial_\epsilon \partial_\zeta F) \left(\tau\epsilon,\cdot, \tau\epsilon \left(\frac{1}{2}I + W_{\partial\Omega^i}\right)[\psi^i] + \zeta^i\right) \right\|_{\COi}
	\\
	&\qquad\leq C_6 \, C_7 \left(1 + \left\|\tau\epsilon \left(\frac{1}{2}I + W_{\partial\Omega^i}\right)[\psi^i] + \zeta^i\right\|_{\COi}\right)^2\\
	&\qquad
	\leq  C_6\, C_7 \,\left(1 + C_3R_1 +|\zeta^i| \right)^2,
	\\
	&\left\|(\partial^2_\zeta F)\left(\tau\epsilon,\cdot,\tau\epsilon \left(\frac{1}{2}I + W_{\partial\Omega^i}\right)[\psi^i] + \zeta^i\right) \right\|_{\COi} 
	\\
	&\qquad\leq C_6\, C_7  \left(1 + \left\|\tau\epsilon \left(\frac{1}{2}I + W_{\partial\Omega^i}\right)[\psi^i]+\zeta^i\right\|_{\COi}\right)^2\\
	&\qquad
	\leq  C_6\, C_7  \left(1 + C_3R_1 +|\zeta^i| \right)^2,
	\end{aligned}	
	\end{equation}
	for all $\tau\in]0,1[$.
	Hence, in view of \eqref{uniccostantR_1R_2}, \eqref{unicboundNe^2FezFz^2F} and \eqref{unicboundezFz^2F} and by \eqref{unicA2equationnorm} and \eqref{unicA3equationnorm} we have
	\begin{equation}\label{unicboundA2A3}
	\begin{aligned}
	&\|T_2[\epsilon, \psi^i, \Psi^i[\epsilon]](\tau,\cdot)\|_{\COi}\\
	 &\quad\leq   \left\{2  C_6 C_7  \left(1 + C_3R_1 +|\zeta^i| \right)^2 + 4 C_3 C_6 C_7 R_2 (1+R)^2 \right\} \delta^\ast\,,
	\\
	&\|T_3[\epsilon, \psi^i, \Psi^i[\epsilon]](\tau,\cdot)\|_{\COi}\\
	&\quad \leq  \left\{ 
	4  C_6 C_7  \left(1 + C_3 R_1 +|\zeta^i| \right)^2 R +  8 C_3^2 C_6 C_7 R_2^2 (1+R)^2
	\right\} \delta^\ast\,,
	\end{aligned}
	\end{equation}
	for all $\tau\in]0,1[$, where to obtain the inequality for $T_3[\epsilon, \psi^i, \Psi^i[\epsilon]](\tau,\cdot)$ we have also used  that 
	\begin{equation*}
	\left\|\left(\frac{1}{2}I + W_{\partial\Omega^i}\right)[\psi^i] + \left(\frac{1}{2}I + W_{\partial\Omega^i}\right)[\Psi^i[\epsilon]] \right\|_{\COi} \leq C_3(R_1 + R_2) \leq R
	\end{equation*}
	(cf. \eqref{uniccostantR}). Moreover, since the boundedness provided in \eqref{unicboundA1} and \eqref{unicboundA2A3} is uniform with respect to $\tau \in ]0,1[$, one verifies that the following inequality holds:
	\begin{equation}\label{intA123}
	\begin{split}
	& \bigg\| \int_{0}^{1} (1-\tau) \left\{  T_1[\epsilon, \psi^i, \Psi^i[\epsilon]](\tau,\cdot) +  2 \, T_2[\epsilon, \psi^i, \Psi^i[\epsilon]](\tau,\cdot)\right.\\
	&\left.\qquad\qquad\qquad\qquad\qquad\qquad\qquad\quad
	+ T_3[\epsilon, \psi^i, \Psi^i[\epsilon]](\tau,\cdot) \right\}  \,d\tau \bigg\|_{\COi}
	\\
	& 
	\leq \int_{0}^{1} (1-\tau) \|  T_1[\epsilon, \psi^i, \Psi^i[\epsilon]](\tau,\cdot) +  2 \, T_2[\epsilon, \psi^i, \Psi^i[\epsilon]](\tau,\cdot)\\
	&\qquad\qquad\qquad\qquad\qquad\qquad\qquad\quad
	+ T_3[\epsilon, \psi^i, \Psi^i[\epsilon]](\tau,\cdot) \|_{\COi} \,d\tau
	\\
	&
	\le \bigg\{ 2 C_6\, C_7\,(1+R)^2 + 4  C_6 C_7  \left(1 + C_3R_1 +|\zeta^i| \right)^2 + 8 C_3 C_6 C_7 R_2 (1+R)^2
	\\
	&\quad +4  C_6 C_7  \left(1 + C_3 R_1 +|\zeta^i| \right)^2 R +  8 C_3^2 C_6 C_7 R_2^2 (1+R)^2 \bigg\}\;\delta^\ast\,.
	\end{split} 
	\end{equation}
	Then, by \eqref{s1-S1<A123} and \eqref{intA123} we obtain 
	\begin{equation}\label{s1-S1}
	\begin{split}
	&\| S_2[\epsilon, \psi^i] - S_2[\epsilon,\Psi^i[\epsilon]] \|_{\COi}\\
	&\qquad \le \bigg\{ 2 C_6\, C_7\,(1+R)^2 + 4  C_6 C_7  \left(1 + C_3R_1 +|\zeta^i| \right)^2 + 8 C_3 C_6 C_7 R_2 (1+R)^2
	\\
	&\qquad \quad +4  C_6 C_7  \left(1 + C_3 R_1 +|\zeta^i| \right)^2 R +  8 C_3^2 C_6 C_7 R_2^2 (1+R)^2 \bigg\}\;\delta^\ast
	\end{split}
	\end{equation}
	(also recall that $\epsilon\in]0,1[$). 
	
	\smallskip
	{\it $\bullet$ Step 7: Conclusion for $S$.}
	\\
	Finally, by \eqref{unicS1inqual}, \eqref{unicS2inequal} and \eqref{s1-S1}, we have
	\begin{equation*}
	\| S[\epsilon, \psi^i] - S[\epsilon,\Psi^i[\epsilon]] \|_{\COi \times C^{0,\alpha}(\partial\Omega^i) \times \COo} \leq C_{8}\; \delta^\ast,
	\end{equation*}
	with
	\[
	\begin{split}
	C_8 \equiv&  \,C_4\, C_5\, (1+ R^\alpha) + 2 C_6\, C_7\,(1+R)^2 + 4  C_6 C_7  \left(1 + C_3R_1 +|\zeta^i| \right)^2 
	\\
	& + 8 C_3 C_6 C_7 R_2 (1+R)^2+4  C_6 C_7  \left(1 + C_3 R_1 +|\zeta^i| \right)^2 R \\
	&+  8 C_3^2 C_6 C_7 R_2^2 (1+R)^2.
	\end{split}
	\]
	
	\smallskip
	{\it $\bullet$ Step 8: Estimate for \eqref{unidensitiesinB0K} and determination of $\delta^\ast$.}
	\\
	By \eqref{N^-1S} and \eqref{unicconstantN-1} we conclude that the norm of the difference between $(\phi^o, \phi^i, \zeta, \psi^i)$ and
	\begin{equation*}
	(\Phi^o[\epsilon] , \Phi^i[\epsilon],	Z[\epsilon] , \Psi^i[\epsilon])
	\end{equation*} 
	in the space $\COo \times \COi_0 \times \R \times \COi$ is less than $C_1C_8\, \delta^\ast$. Then, by \eqref{PhiBK/2} and   by the triangular inequality we obtain 
	\[
	\|(\phi^o, \phi^i, \zeta, \psi^i) - (\phi^o_0, \phi^i_0, \zeta_0, \psi^i_0)\|_{\COo \times \COi_0 \times \R \times \COi} \leq C_1C_{8}\, \delta^\ast + \frac{K}{2} \,.
	\]
	Thus, in order to have $(\phi^o, \phi^i, \zeta, \psi^i) \in \mathcal{B}_{0,K}$, it suffices to take 
	\[
	\delta^\ast<\frac{K}{2C_1C_{8}}
	\] 
    in inequality \eqref{unidisuguiesp}. Then, for such choice of $\delta^\ast$, \eqref{unidensitiesinB0K} holds and the theorem is proved.
\end{proof}

	\subsection{Local uniqueness for the family of solutions.}\label{family} As a consequence of Theorem \ref{Thmunisol2}, we can derive the following local uniqueness result for the family $\{(u^o_\epsilon,u^i_\epsilon)\}_{\epsilon \in ]0,\epsilon'[}$.

	\begin{cor}\label{cor}	Let assumptions  \eqref{zetaicond}, \eqref{F1}, and \eqref{realanalhp} hold true.  Let $\epsilon' \in ]0,\epsilon_0[$ be as in Theorem \ref{existenceThm} (v).	Let $\{(u^o_\epsilon,u^i_\epsilon)\}_{\epsilon \in ]0,\epsilon'[}$ be as in Theorem \ref{uesol}. Let $\{(v^o_\epsilon,v^i_\epsilon)\}_{\epsilon \in ]0,\epsilon'[}$ be a family of functions such that  
		$(v^o_\epsilon,v^i_\epsilon) \in C^{1,\alpha}(\Omega(\epsilon)) \times C^{1,\alpha}(\epsilon\Omega^i)$ is a solution of problem \eqref{princeq} for all $\epsilon\in]0,\epsilon'[$. If 
		\begin{equation}\label{cor.eq1}
		\lim_{\epsilon\to 0^+}\epsilon^{-1}\left\| {v^i_\epsilon(\epsilon \cdot ) - u^i_\epsilon(\epsilon \cdot )} \right\| _{C^{1,\alpha}(\partial \Omega^i)}=0,
		\end{equation}
		then there exists $\epsilon_\ast\in]0,\epsilon'[$ such that 
		\[
		(v^o_\epsilon,v^i_\epsilon)=(u^o_\epsilon,u^i_\epsilon)\quad\forall \epsilon\in]0,\epsilon_\ast[\,.
		\]
	\end{cor}
	\begin{proof} Let $\epsilon^*$ and $\delta^*$ be as in Theorem \ref{Thmunisol2}.  By \eqref{cor.eq1} there is $\epsilon_*\in]0,\epsilon^*[$ such that 
		\[
		\left\| {v^i_\epsilon(\epsilon \cdot ) - u^i_\epsilon(\epsilon \cdot )} \right\| _{C^{1,\alpha}(\partial \Omega^i)}\le \epsilon\delta^*\quad\forall\epsilon\in]0,\epsilon_*[.
		\]
		Then the statement follows by Theorem \ref{Thmunisol2}.
	\end{proof}

\section{Appendix}

In this Appendix, we present some technical facts, which have been exploited in this paper, on product and composition of functions of $C^{0,\alpha}$ and $C^{1,\alpha}$ regularity. We start by introducing the following elementary result (cf. Lanza de Cristoforis \cite{La91}).

\begin{lemma}\label{C0C1alphaalgebra}
	Let $n \in \N\setminus\{0\}$. Let $\Omega$ be a bounded connected open subset of $\R^n$ of class $C^{1,\alpha}$. Then
	\begin{equation*}
	\|uv\|_{C^{0,\alpha}(\partial\Omega)} \leq \|u\|_{C^{0,\alpha}(\partial\Omega)} \, \|v\|_{C^{0,\alpha}(\partial\Omega)} \qquad \forall u,v \in C^{0,\alpha}(\partial\Omega),
	\end{equation*}
	and 
	\begin{equation*}
	\|uv\|_{C^{1,\alpha}(\partial\Omega)} \leq 2 \, \|u\|_{C^{1,\alpha}(\partial\Omega)} \, \|v\|_{C^{1,\alpha}(\partial\Omega)} \qquad \forall u,v \in C^{1,\alpha}(\partial\Omega).
	\end{equation*}
\end{lemma}

We now present a result on composition of a $C^{m,\alpha}$ function, with $m\in \{0,1\}$, with a $C^{1,\alpha}$ function.

\begin{lemma}\label{lemmacompC1alpha}
	Let $n,d \in \N \setminus \{0\}$ and $\alpha \in \, ]0,1]$. Let $\Omega_1$ be an open bounded convex subset of $\R^n$ and $\Omega_2$ be an open bounded convex subset of $\R^d$. Let $v = (v_1,\dots,v_n) \in (C^{1,\alpha}(\overline{\Omega_2}))^n$ such that 
	$v(\overline{\Omega_2}) \subset \overline{\Omega_1}$. Then the following statement holds.
	\begin{enumerate}
		\item[(i)] If $u \in C^{0,\alpha}(\overline{\Omega_1})$, then
		\begin{equation*}
		\|u(v(\cdot))\|_{C^{0,\alpha}(\overline{\Omega_2})} \leq \|u\|_{C^{0,\alpha}(\overline{\Omega_1})} 
		\left(1 + \|v\|^\alpha_{(C^{1,\alpha}(\overline{\Omega_2}))^n} \right).
		\end{equation*}
		
		\item[(ii)] If $u \in C^{1,\alpha}(\overline{\Omega_1})$, then
		\begin{equation*}
		\|u(v(\cdot))\|_{C^{1,\alpha}(\overline{\Omega_2})} 
		\leq
		(1 + nd)^2 \|u\|_{C^{1,\alpha}(\overline{\Omega_1})} \left(1 + \, \|v\|_{(C^{1,\alpha}(\overline{\Omega_2}))^n}  \right)^2.
		\end{equation*}  
	\end{enumerate}
\end{lemma}

Then, by Lemma \ref{lemmacompC1alpha}, we deduce the following Proposition \ref{proponcompC1alpha}.

\begin{prop}\label{proponcompC1alpha}
	Let $n \in \N \setminus \{0\}$. Let $\alpha \in ]0,1]$. Let $\Omega$ be a bounded connected open subset of $\R^n$ of class $C^{1,\alpha}$. Let $R>0$. Then the following holds.
	\begin{enumerate}
		\item[(i)] There exists $c_0>0$ such that 
		\begin{equation*}
		\|u(\cdot,v(\cdot))\|_{C^{0,\alpha}(\partial\Omega)} \leq c_0  \|u\|_{C^{0,\alpha}(\partial\Omega \times [-R,R])} \left(1 + \|v\|^\alpha_{C^{1,\alpha}(\partial\Omega)} \right).
		\end{equation*}
		for all $u \in C^{0,\alpha}(\partial\Omega \times \R)$ and for all $v \in C^{1,\alpha}(\partial \Omega)$ such that $v(\partial \Omega) \subset [-R,R]$. 	
		
		\item[(ii)] There exists $c_1>0$ such that 
		\begin{equation*}
		\|u(\cdot,v(\cdot))\|_{C^{1,\alpha}(\partial\Omega)} \leq c_1  \|u\|_{C^{1,\alpha}(\partial\Omega \times [-R,R])} \left(1 + \|v\|_{C^{1,\alpha}(\partial\Omega)} \right)^2.
		\end{equation*}
		for all $u \in C^{1,\alpha}(\partial\Omega \times \R)$ and for all $v \in C^{1,\alpha}(\partial \Omega)$ such that $v(\partial \Omega) \subset [-R,R]$. 
	\end{enumerate}
\end{prop}

\begin{proof}
	We prove only statement (ii). The proof of statement (i) can be obtained adapting the one of point (ii) and using Lemma \ref{lemmacompC1alpha} (i) instead of Lemma \ref{lemmacompC1alpha} (ii).
	
	Since $\partial\Omega$ is compact and of class $C^{1,\alpha}$, a standard argument shows that there are a finite cover of $\partial \Omega$ consisting of open subsets $\mathcal{U}_1,\dots,\mathcal{U}_k$  of $\partial\Omega$ and for each $j\in\{1,\dots,k\}$ a $C^{1,\alpha}$ diffeomorphism $\gamma_j$ from $\overline{B_{n-1}(0,1)}$ to the closure of $\mathcal{U}_j$ in $\partial\Omega$.
	Then, for a fixed $j \in \{1,\dots,k\}$ we define the functions $\tilde{u}^j: \overline{B_{n-1}(0,1)} \times \R \to \R$, $\tilde{v}^j: \overline{B_{n-1}(0,1)} \to \R$, and $\tilde w^j: \overline{B_{n-1}(0,1)} \to \overline{B_{n-1}(0,1)} \times \R$ by setting
	\begin{equation*}
		\begin{aligned}
			&\tilde{u}^j(t',s) \equiv u(\gamma_j(t'),s) \qquad &&\forall(t',s) \in \overline{B_{n-1}(0,1)} \times \R,
			\\
			&\tilde{v}^j(t') \equiv v(\gamma_j(t')) \qquad &&\forall t' \in \overline{B_{n-1}(0,1)},
			\\
			&\tilde{w}^j(t') \equiv (t',\tilde{v}^j(t')) \qquad 
			&&\forall t' \in \overline{B_{n-1}(0,1)}.
		\end{aligned}
	\end{equation*}
	Since $v(\partial\Omega) \subset [-R,R]$, it follows that $\tilde{w}^j(\overline{B_{n-1}(0,1)}) \subset \overline{B_{n-1}(0,1)} \times [-R,R]$. Moreover,   we can see that there exists $d_j>0$ such that
	\begin{equation*}
		\|\tilde{w}^j\|_{(C^{1,\alpha}(\overline{B_{n-1}(0,1)}))^{n}} \leq d_j \|\tilde{v}^j\|_{\left(C^{1,\alpha}(\overline{B_{n-1}(0,1)})\right)}. 
	\end{equation*} 
	Then Lemma \ref{lemmacompC1alpha} implies that there exists $c_j>0$ such that
	\begin{equation}\label{proponcompC1alpha.eq1}
		\begin{split}
			\|\tilde{u}^j(\cdot,\tilde{v}^j(\cdot))\|&_{C^{1,\alpha}(\overline{B_{n-1}(0,1)})}  = \|\tilde{u}^j(\tilde{w}^j(\cdot))\|_{(C^{1,\alpha}(\overline{B_{n-1}(0,1)}))^n} 
			\\
			&\leq c_j \|\tilde{u}^j\|_{C^{1,\alpha}(\overline{B_{n-1}(0,1)}\times [-R,R])} \left(1 + \|\tilde{w}^j\|_{(C^{1,\alpha}(\overline{B_{n-1}(0,1)}))^n}\right)^2
			\\
			&\leq c_j \|\tilde{u}^j\|_{C^{1,\alpha}(\overline{B_{n-1}(0,1)}\times [-R,R])} \left(1 + d_j \|\tilde{v}^j\|_{C^{1,\alpha}(\overline{B_{n-1}(0,1)})}\right)^2.
		\end{split}
	\end{equation}
	Without loss of generality, we can now assume that the norm of $C^{1,\alpha}(\partial\Omega)$ is defined on the atlas $\{(\mathcal{U}_j,\gamma_j)\}_{j\in\{1,\dots,k\}}$ (cf.~Section \ref{notation}). Then by \eqref{proponcompC1alpha.eq1} we have
	\begin{equation}\label{proponcompC1alpha.eq2}
		\begin{split}
			\|u(\cdot,v(\cdot))\|_{C^{1,\alpha}(\partial\Omega)}&=\sum_{j=1}^k\|u(\gamma_j(\cdot),v(\gamma_j(\cdot)))\|_{C^{1,\alpha}(\overline{B_{n-1}(0,1)})}\\
			&=\sum_{j=1}^k\|\tilde{u}^j(\cdot,\tilde{v}^j(\cdot))\|_{C^{1,\alpha}(\overline{B_{n-1}(0,1)})}\\
			&\le \sum_{j=1}^kc_j \|\tilde{u}^j\|_{C^{1,\alpha}(\overline{B_{n-1}(0,1)}\times [-R,R])} \left(1 + d_j \|\tilde{v}^j\|_{C^{1,\alpha}(\overline{B_{n-1}(0,1)})}\right)^2\,.
		\end{split}
	\end{equation}
	Moreover,
	\[
	\|\tilde{u}^j\|_{C^{1,\alpha}(\overline{B_{n-1}(0,1)}\times [-R,R])}\leq \|u\|_{C^{1,\alpha}(\partial\Omega \times [-R,R] )}
	\]
	and 
	\[
	\begin{split}
	&\left(1 + d_j \|\tilde{v}^j\|_{C^{1,\alpha}(\overline{B_{n-1}(0,1)})}\right)^2\\
	&\quad\le(1+d_j)^2\left(1 + \|\tilde{v}^j\|_{C^{1,\alpha}(\overline{B_{n-1}(0,1)})}\right)^2\le (1+d_j)^2\left(1 + \|v\|_{C^{1,\alpha}(\partial\Omega)} \right)^2.
	\end{split}
	\]
	Hence, \eqref{proponcompC1alpha.eq2} implies that
	\[
	\begin{split}
	&\|u(\cdot,v(\cdot))\|_{C^{1,\alpha}(\partial\Omega)}\\
	&\quad\le k\, \max \{ c_1(1+d_1)^2,\dots,c_k(1+d_k)^2 \} \|u\|_{C^{1,\alpha}(\partial\Omega \times [-R,R] )} \left(1 + \|v\|_{C^{1,\alpha}(\partial\Omega)} \right)^2
	\end{split}
	\]
	and the proposition is proved.
\end{proof}

\section*{Acknowledgement}
The research of the R.~Molinarolo was supported by HORIZON 2020 RISE project ``MATRIXASSAY'' under project number 644175, during the author's secondment at the University of Texas at Dallas. The author gratefully acknowledges the University of Texas at Dallas and the University of Tulsa for the great research environment and the friendly atmosphere provided. P.~Musolino is member of the `Gruppo Nazionale per l'Analisi Matematica, la Probabilit\`a e le loro Applicazioni' (GNAMPA) of the `Istituto Nazionale di Alta Matematica' (INdAM).









\end{document}